\documentclass[12pt]{article}

\usepackage[square,sort,comma,numbers]{natbib}
\usepackage{times} 
\usepackage{latexsym,epsfig,amssymb,amsmath,amsfonts,graphicx,amsthm,ifthen,pifont,comment,enumerate,setspace,multirow,color}
\usepackage{mhequ}
\usepackage{multicol,booktabs,colortbl,tabularx}

\usepackage{JASA_manu}

\def\begmat{\left(\begin{array}}\def\endmat{\end{array}\right)}

\def\bi{\begin{itemize}\setlength{\itemsep}{0pt}} \def\ei{\end{itemize}}

\def\bl{\begin{list}{\labelitemi}{\leftmargin=1em}\setlength{\itemsep}{-2.5pt}}  \def\el{\end{list}}
\def\bn{\begin{enumerate}} \def\en{\end{enumerate}}
\def\bt{\begin{table}[h]} \def\et{\end{table}}
\def\bc{\begin{center}} \def\ec{\end{center}}
\def\T{{ \mathrm{\scriptscriptstyle T} }}

\newcommand{\abs}[1]{\left\vert#1\right\vert}

\newcommand{\norm}[1]{\left\Vert#1\right\Vert}

\newcommand \bbP{\mathbb{P}}
\newcommand \bbE{\mathbb{E}}

\newtheorem{theorem}{Theorem}[section]
\newtheorem{lemma}[theorem]{Lemma}
\newtheorem{proposition}[theorem]{Proposition}

\newtheorem{remark}[theorem]{Remark}

\theoremstyle{plain}

\theoremstyle{plain}

\theoremstyle{remark}

\theoremstyle{plain}

\newcommand \be{\begin{equs}}
\newcommand \ee{\end{equs}}

\begin{document}

\title{Bayesian shrinkage}
\author{Anirban Bhattacharya, Debdeep Pati, Natesh S. Pillai, David B. Dunson }
%
\maketitle

\begin{center}
\textbf{Abstract}
\end{center}
Penalized regression methods, such as $L_1$ regularization, are routinely used in high-dimensional applications, and there is a rich literature on optimality properties under sparsity assumptions.  In the Bayesian paradigm, sparsity is routinely induced through two-component mixture priors having a probability mass at zero, but such priors encounter daunting computational problems in high dimensions.  This has motivated an amazing variety of continuous shrinkage priors, which can be expressed as global-local scale mixtures of Gaussians, facilitating computation.  In sharp contrast to the corresponding frequentist literature, very little is known about the properties of such priors.  Focusing on a broad class of shrinkage priors, we provide precise results on prior and posterior concentration.  Interestingly, we demonstrate that most commonly used shrinkage priors,  including the Bayesian Lasso, are suboptimal in high-dimensional settings.  A new class of Dirichlet Laplace (DL) priors are proposed, which are optimal and lead to efficient posterior computation exploiting results from normalized random measure theory.  Finite sample performance of Dirichlet Laplace priors relative to alternatives is assessed in simulations.
\vspace*{.3in}


\section{Introduction}

High-dimensional data have become commonplace in broad application areas, and there is an exponentially increasing literature on statistical and computational methods for big data.  In such settings, it is well known that classical methods such as maximum likelihood estimation break down, motivating a rich variety of alternatives based on penalization and thresholding.  Most penalization approaches produce a point estimate of a high-dimensional coefficient vector, which has a Bayesian interpretation as corresponding to the mode of a posterior distribution obtained under a shrinkage prior.  For example, the wildly popular Lasso/$L_1$ regularization approach to regression \cite{tibshirani1996regression} is equivalent to maximum {\em a posteriori} (MAP) estimation under a Gaussian linear regression model having a double exponential (Laplace) prior on the coefficients.  There is a rich theoretical literature justifying the optimality properties of such penalization approaches \cite{zhao2007model,van2008high,zhang2008sparsity,meinshausen2009lasso,raskutti2011minimax,negahban2010unified}, with fast algorithms \cite{efron2004least} and compelling applied results leading to routine use of $L_1$ regularization in particular.

The overwhelming emphasis in this literature has been on rapidly producing a point estimate with good empirical and theoretical properties.  However, in many applications, 
it is crucial to be able to obtain a realistic characterization of uncertainty in the parameters, in functionals of the parameters and in predictions.  Usual frequentist approaches to characterize uncertainty, such as constructing asymptotic confidence regions or using the bootstrap, can break down in high-dimensional settings.  For example, in regression when the number of subjects $n$ is much less than the number of predictors $p$, one cannot naively appeal to asymptotic normality and resampling from the data may not provide an adequate characterization of uncertainty.  

Given that most shrinkage estimators correspond to the mode of a Bayesian posterior, it is natural to ask whether we can use the whole posterior distribution to provide a probabilistic measure of uncertainty.  Several important questions then arise.  Firstly, from a frequentist perspective, we would like to be able to choose a default shrinkage prior that leads to similar optimality properties to those shown for $L_1$ penalization and other approaches.  However, instead of showing that a particular penalty leads to a point estimator having a minimax optimal rate of convergence under sparsity assumptions, we would like to obtain a (much stronger) result that the entire posterior distribution concentrates at the optimal rate, i.e., the posterior probability assigned to a shrinking neighborhood (proportionally to the optimal rate) of the true value of the parameter  converges to one.  In addition to providing a characterization of uncertainty, taking a Bayesian perspective has distinct advantages in terms of tuning parameter choice, allowing key penalty parameters to be marginalized over the posterior distribution instead of relying on cross-validation.  Also, by inducing penalties through shrinkage priors, important new classes of penalties can be discovered that may outperform usual $L_q$-type choices.

An amazing variety of shrinkage priors have been proposed in the Bayesian literature, with essentially no theoretical justification for the performance of these priors in the high-dimensional settings for which they were designed.  \cite{ghosal1999asymptotic} and \cite{bontemps2011bernstein} provided conditions on the prior for asymptotic normality of linear regression coefficients allowing the number of predictors $p$ to increase with sample size $n$, with \cite{ghosal1999asymptotic} requiring a very slow rate of growth and \cite{bontemps2011bernstein} assuming $p \le n$.  These
results required the prior to be sufficiently flat in a neighborhood of the true parameter value, essentially ruling out shrinkage priors.  \cite{armagan2011generalized} considered shrinkage priors in providing simple sufficient conditions for posterior consistency in $p \le n$ settings, while \cite{strawn2012finite} studied finite sample posterior contraction in $p \gg n$ settings.  

In studying posterior contraction in high-dimensional settings, it becomes clear that it is critical to obtain tight bounds on prior concentration.  This substantial technical hurdle has prevented any previous results (to our knowledge) on posterior concentration in $p \gg n$ settings for shrinkage priors.  In fact, prior concentration is critically important not just in studying frequentist optimality properties of Bayesian procedures but for Bayesians in obtaining a better understanding of the behavior of their priors.  Without a precise handle on prior concentration, Bayesians are operating in the dark in choosing shrinkage priors and the associated hyperparameters.  It becomes an art to use intuition and practical experience to indirectly induce a shrinkage prior, while focusing on Gaussian scale families for computational tractability.  Some beautiful classes of priors have been proposed by \cite{griffin2010inference,carvalho2010horseshoe,armagan2011generalized} among others, with \cite{polson2010shrink} showing that essentially all existing shrinkage priors fall within the Gaussian global-local scale mixture family.  One of our primary goals is to obtain theory that can allow evaluation of existing priors and design of novel priors, which are appealing from a Bayesian perspective in allowing incorporation of prior knowledge and from a frequentist perspective in leading to minimax optimality under weak sparsity assumptions. 

Shrinkage priors provide a continuous alternative to point mass mixture priors, which include a mass at zero mixed with a continuous density.  These priors are highly appealing in allowing separate control of the level of sparsity and the size of the signal coefficients.  In a beautiful recent article, \cite{castilloneedles} showed optimality properties for carefully chosen point mass mixture priors in high-dimensional settings.  Unfortunately, such priors lead to daunting computational hurdles in high-dimensions due to the need to explore a $2^p$ model space; an NP-hard problem.  Continuous scale mixtures of Gaussian priors can potentially lead to dramatically more efficient posterior computation.  

Focusing on the normal means problem for simplicity in exposition, we provide general theory on prior and posterior concentration under shrinkage priors.  One of our main results is that a broad class of Gaussian scale mixture priors, including the Bayesian Lasso \cite{park2008bayesian} and other commonly used choices such as ridge regression, are sub-optimal.  We provide insight into the reasons for this sub-optimality and propose a new class of Dirichlet-Laplace (DL) priors, which are optimal and lead to efficient posterior computation.  We show promising initial results for DL and Dirichlet-Cauchy (DC) priors relative to a variety of competitors.

\section{Preliminaries} \label{sec:prelim}

In studying prior and posterior computation for shrinkage priors, we require some notation and technical concepts.  We introduce some of the basic notation here.  Technical details in the text are kept to a minimum, and proofs are deferred to a later section.   

Given sequences $a_n, b_n$, we denote $a_n = O(b_n)$ if there exists a global constant $C$ such that $a_n \leq C b_n$ and $a_n = o(b_n)$ if $a_n/b_n \to 0$ as $n \to \infty$.  For a vector $ x \in \mathbb{R}^r$,  $\norm{x}_2$ denotes its Euclidean norm. We will use $\Delta^{r-1}$ to denote the $(r-1)$-dimensional simplex $\{ x = (x_1, \ldots, x_{r})^{\T} : x_j \geq 0, \sum_{j=1}^{r} x_j = 1\}$. Further, let $\Delta_0^{r-1}$ denote $\{ x = (x_1, \ldots, x_{r-1})^{\T} : x_j \geq 0, \sum_{j=1}^{r-1} x_j \leq 1\}$.

For a subset $S \subset \{1, \ldots, n\}$, let $|S|$ denote the cardinality of $S$ and define $\theta_S = (\theta_j : j \in S)$ for a vector $\theta \in \mathbb{R}^n$. 
Denote $\mbox{supp}(\theta)$ to be the \emph{support} of $\theta$, the subset of $\{1, \ldots, n\}$ corresponding to the non-zero entries of $\theta$. 
Let $l_0[q; n]$ denote the subset of  $\mathbb{R}^n$ given by
\begin{align*}
l_0[q;n] = \{ \theta \in \mathbb{R}^n ~:~ \#(1 \leq j \leq n : \theta_j \neq 0) \leq q\}.
\end{align*}
Clearly, $l_0[q; n]$ consists of $q$-sparse vectors $\theta$ with $|\mbox{supp}(\theta)| \leq q$.

Let $\mbox{DE}(\tau)$ denote a zero mean double-exponential or Laplace distribution with density $f(y) = (2 \tau)^{-1} e^{- \abs{y}/\tau}$ for $y \in \mathbb{R}$.  
Also, we use the following parametrization for the three-parameter generalized inverse Gaussian (giG) distribution: $Y \sim \mbox{giG}(\lambda, \rho, \chi)$ if $f(y) \propto y^{\lambda -1} e^{ - 0.5 (\rho y + \chi/y)}$ for $y > 0$. 

\section{Concentration properties of global-local priors} \label{sec:conc_prop}

\subsection{\bf{Motivation}}

For a high-dimensional vector $\theta \in \mathbb{R}^n$, a natural way to incorporate sparsity in a Bayesian framework is to use point mass mixture priors 
\begin{align}\label{eq:point_mass}
\theta_j \sim (1 - \pi) \delta_0 + \pi g_{\theta}, \quad j = 1, \ldots, n ,
\end{align}
where $\pi = \mbox{Pr}(\theta_j \neq 0)$, $\bbE\{ |\mbox{supp}(\theta)|  \mid \pi\} = n \pi$ is the prior guess on model size (sparsity level), and $g_{\theta}$ is an absolutely continuous density on $\mathbb{R}$.  It is common to place a beta prior on $\pi$, leading to a beta-Bernoulli prior on the model size, which conveys an automatic multiplicity adjustment \citep{scott2010bayes}.  \cite{castilloneedles} established that prior (\ref{eq:point_mass}) with an appropriate beta prior on $\pi$ and suitable tail conditions on $g_{\theta}$ leads to a frequentist minimax optimal rate of posterior contraction in the normal means setting.  We shall revisit the normal means problem in subsection \ref{subsec:norm_means}.

Although point mass mixture priors are intuitively appealing and possess attractive theoretical properties, posterior sampling requires a stochastic search over an enormous space in complicated models where marginal likelihoods are not available analytically, leading to slow mixing and convergence \cite{polson2010shrink}. Computational issues and considerations that many of the $\theta_j$s may be small but not exactly zero has motivated a rich literature on continuous shrinkage priors; for some flavor of the vast literature refer to \citep{park2008bayesian,carvalho2010horseshoe,griffin2010inference,hans2011elastic,armagan2011generalized}.
\citep{polson2010shrink} noted that essentially all such shrinkage priors can be represented as global-local (GL) mixtures of Gaussians, 
\begin{align}\label{eq:lg}
\theta_j \sim \mbox{N}(0, \psi_j \tau), \quad \psi_j \sim f, \quad \tau \sim g,  
\end{align}
where $\tau$ controls global shrinkage towards the origin while the local scales $\{ \psi_j \}$ allow deviations in the degree of shrinkage.  If $g$ puts sufficient mass near zero and $f$ is appropriately chosen, GL priors in (\ref{eq:lg}) can intuitively approximate (\ref{eq:point_mass}) but through a continuous density concentrated near zero with heavy tails. 

GL priors potentially have substantial computational advantages over variable selection priors, since the normal scale mixture representation allows for conjugate updating of $\theta$ and $\psi$ in a block. Moreover, a number of frequentist regularization procedures such as ridge, lasso, bridge and elastic net correspond to posterior modes under GL priors with appropriate choices of $f$ and $g$. For example, one obtains a double-exponential prior corresponding to the popular $L_1$ or lasso penalty if $f$ has an exponential distribution. However, unlike variable selection priors (\ref{eq:point_mass}), many aspects of shrinkage priors are poorly understood.  For example, even basic properties, such as how the prior concentrates around an arbitrary sparse $\theta_0$, remain to be shown. Hence, Bayesians tend to operate in the dark in using such priors, and frequentists tend to be skeptical due to the lack of theoretical justification.

This skepticism is somewhat warranted, as it is clearly the case that reasonable seeming priors can have poor performance in high-dimensional settings.  For example, choosing $\pi=1/2$ in prior (\ref{eq:point_mass}) leads to an exponentially small prior probability of $2^{-n}$ assigned to the null model, so that it becomes literally impossible to override that prior informativeness with the information in the data to pick the null model. However, with a beta prior on $\pi$, this problem can be avoided \citep{scott2010bayes}.  In the same vein, if one places i.i.d. $\mbox{N}(0, 1)$ priors on the entries of $\theta$, then the induced prior on $\norm{\theta}$ is highly concentrated around $\sqrt{n}$ leading to misleading inferences on $\theta$ almost everywhere.  These are simple cases, but it is of key importance to assess whether such problems arise for other priors in the GL family and if so, whether improved classes of priors can be found.

There has been a recent awareness of these issues, motivating a basic assessment of the marginal properties of shrinkage priors for a single $\theta_j$. Recent priors such as the horseshoe \cite{carvalho2010horseshoe} and generalized double Pareto \cite{armagan2011generalized} are carefully formulated to obtain marginals having a high concentration around zero with heavy tails.  This is well justified, but as we will see below, such marginal behavior alone is not sufficient; it is necessary to study the joint distribution of $\theta$ on $\mathbb{R}^n$.  Specifically, we recommend studying the prior concentration $\bbP(\norm{\theta - \theta_0} < t_n)$ where the true parameter $\theta_0$ is assumed to be sparse: $\theta_0 \in l_0[q_n; n]$ with the number of non-zero components $q_n \ll n$ and
\begin{align}\label{eqn:tn}
t_n = n^{\delta/2} \quad \mathrm{with} \quad \delta \in (0,1).  
\end{align}

In models where $q_n \ll n$,  the prior must place sufficient mass around sparse vectors to allow for good posterior contraction; see subsection \ref{subsec:norm_means} for further details. Now, as a first illustration, consider the following two extreme scenarios: i.i.d. standard normal priors for the individual components $\theta_j$ \textit{vs.} point mass mixture priors given by \eqref{eq:point_mass}.


\begin{theorem}\label{thm:basic_stmt}
Assume that $\theta_0 \in l_0[q_n; n]$ with $q_n = o(n)$. Then, for i.i.d standard normal priors on $\theta_j$, 
\begin{align}\label{thm:norm_conc}
\bbP(\norm{\theta - \theta_0}_2 < t_n) \leq e^{-c\, n}. 
\end{align}
For point mass mixture priors \eqref{eq:point_mass} with $\pi \sim \mbox{Beta}(1, n + 1)$ and $g_{\theta}$ being a standard Laplace distribution $g_{\theta} \equiv \mbox{DE}(1)$, 
\begin{align}\label{thm:pm_conc}
\bbP(\norm{\theta - \theta_0}_2 < t_n) \geq e^{-c \max\{q_n, \norm{\theta_0}_1\}}.
\end{align}
\end{theorem}
\begin{proof}
Using $\norm{\theta}_2^2 \sim \chi_n^2$, the claim made in \eqref{thm:norm_conc} follows from an application of Anderson's inequality \eqref{lem:anderson} and standard chi-square deviation inequalities. In particular, the exponentially small concentration also holds for $\bbP(\norm{\theta_0}_2 < t_n)$. The second claim \eqref{thm:pm_conc} follows from results in \cite{castilloneedles}. 
\end{proof}
As seen from Theorem \ref{thm:basic_stmt}, the point mass mixture priors have much improved concentration around sparse vectors, as compared to the i.i.d. normal prior distributions. The theoretical properties enjoyed by the point mass mixture priors can mostly be attributed to this improved concentration. The above comparison suggests that it is of merit to evaluate a shrinkage prior in high dimensional models under sparsity assumption by obtaining its concentration rates around sparse vectors. In this paper, we carry out this program for a wide class of shrinkage priors. Our analysis also suggests some novel priors with improved concentration around sparse vectors. 

In order to communicate our main results to a wide audience, we will first present specific corollaries of our main results applied to various existing shrinkage priors. The main results are given in Section \ref{sec:main_pf}. Recall the GL priors presented in \eqref{eq:lg} and the sequence $t_n$ in \eqref{eqn:tn}.

\subsection{\bf{Prior concentration for global priors}}


This simplified setting involves  only a global parameter, \textit{i.e.}, $\psi_j = 1$ for all $j$. This subclass includes the important example of ridge regression, with $\tau$ routinely assigned an inverse-gamma prior, $\tau \sim \mbox{IG}(\alpha, \beta)$.  
\begin{theorem}\label{thm:conc_invgam}
Assume $\theta \sim \mathrm{GL}$ with $\psi_j = 1$ for all $j$. If the prior $f$ on the global parameter $\tau$ has an $\mathrm{IG}(\alpha, \beta)$ distribution, then
\begin{align}\label{eq:conc_ig}
\bbP(\norm{\theta}_2 < t_n) \leq e^{- C n^{1- \delta} },
\end{align}
where $C > 0$ is a constant depending only on $\alpha$ and $\beta$.  
\end{theorem}
The above theorem shows that compared to i.i.d. normal priors \eqref{thm:norm_conc}, the prior concentration does not improve much under an inverse-gamma prior on the global variance regardless of the hyperparameters (provided they don't scale with $n$) even when $\theta_0=0$.  Concentration around $\theta_0$ away from zero will clearly be even worse.  Hence, such a prior is not well-suited in high-dimensional settings, confirming empirical observations  documented in \cite{gelman2006prior,polson2011half}. It is also immediate that the same concentration bound in \eqref{eq:conc_ig} would be obtained for the giG family of priors on $\tau$. 

In \cite{polson2011half}, the authors instead recommended a half-Cauchy prior as a default choice for the global variance (also see \cite{gelman2006prior}).  We consider the following general class of densities on $(0, \infty)$ for $\tau$, to be denoted $\mathcal{F}$ henceforth, that satisfy:  (i) $f(\tau) \leq M$ for all $\tau \in (0, \infty)$  (ii) $f(\tau) > 1/M$ for all $\tau \in (0, 1)$, for some constant $M > 0$. Clearly, $\mathcal{F}$ contains the half-Cauchy and exponential families. The following result provides concentration bounds for these priors. 

\begin{theorem}\label{thm:conc_g}
Let $\norm{\theta_0}_2 = o(\sqrt{n})$. If the prior $f$ on the global parameter $\tau$ belongs to the class $\mathcal{F}$ above then,
\begin{align}
& C_1 e^{-(1-\delta) \log n} \leq \bbP(\norm{\theta}_2 <  t_n ) \leq C_2 e^{- (1- \delta) \log n} \label{eq:conc_gonly_c}.
\end{align}
Furthermore, if $\norm{\theta_0}_2 > {t_n}$, then
\begin{align}
e^{- c_1 n \log a_n }\leq \bbP(\norm{\theta - \theta_0}_2 <  t_n ) \leq  e^{- c_2 n \log a_n }, \label{eq:conc_gonly_nc}
\end{align}
where $a_n = \norm{\theta_0}_2/t_n > 1$ and $c_i, C_i > 0$ are constants with $C_1, C_2, c_2$ depending only on $M$ in the definition of $\mathcal{F}$ and $c_1$ depending on $M$ and $\delta$. 
\end{theorem}

Thus \eqref{eq:conc_gonly_c} in Theorem \ref{thm:conc_g} shows that the prior concentration around zero can be dramatically improved from exponential to polynomial with a careful prior on $\tau$ that can assign sufficient mass near zero, such as the half-Cauchy prior \cite{gelman2006prior, polson2011half}.  Unfortunately, as \eqref{eq:conc_gonly_nc} shows,  for signals of large magnitude one again obtains an exponentially decaying probability.  Hence, Theorem \ref{thm:conc_g} conclusively shows that global shrinkage priors are simply not flexible enough for high-dimensional problems. 
\begin{remark}\label{rem:thm:conc_g}
The condition $\norm{\theta_0}_2 \geq t_n$ is only used to prove the lower bound in \eqref{eq:conc_gonly_nc}. For any $\norm{\theta_0}$ bounded below by a constant, we would still obtain an upper bound $e^{- C n^{1- \delta} \log n}$ in \eqref{eq:conc_gonly_nc}, similar to the bound in \eqref{eq:conc_ig}. 
\end{remark}

%
%
%

\subsection{\bf{Prior concentration for a class of GL priors}}
Proving concentration results for the GL family \eqref{eq:lg} in the general setting presents a much harder challenge compared to Theorem \ref{thm:conc_g} since we now have to additionally integrate over the $n$ local parameters $\psi = (\psi_1, \ldots, \psi_n)$.  We focus on an important sub-class in Theorem \ref{thm:conc_lg} below, namely the exponential family for the distribution of $g$ in \eqref{eq:lg}. For analytical tractability, we additionally assume that $\theta_0$ has only one non-zero entry. The interest in the exponential family arises from the fact that normal-exponential scale mixtures give rise to the double-exponential family \cite{west1987scale}: $\theta \mid \psi \sim N(0, \psi \sigma^2), \psi \sim \mbox{Exp}(1/2)$ implies $\theta \sim \mbox{DE}(\sigma)$, and hence this family of priors can be considered as a Bayesian version of the lasso \cite{park2008bayesian}. We now state a concentration result for this class noting that a general version of Theorem \ref{cor:thm_lg} can be found in Theorem \ref{thm:conc_lg} stated in Section \ref{sec:main_pf}. 

\begin{theorem}\label{cor:thm_lg}
Assume $\theta \sim \mathrm{GL}$ with $f \in \mathcal{F}$ and $g \equiv \mathrm{Exp}(\lambda)$ for some constant $\lambda > 0$. Also assume $\theta_0$ has only one non-zero entry and $\norm{\theta_0}_2^2 > \log n$. Then, for a global constant $C > 0$ depending only on $M$ in the definition of $\mathcal{F}$,
\begin{align}\label{thm:lg_ub_cor}
\bbP(\norm{\theta - \theta_0}_2 < t_n) \leq e^{- C \sqrt{n}} . 
\end{align}
\end{theorem}
Theorem \ref{cor:thm_lg} asserts that even in the simplest deviation from the null model with only one signal, one continues to have exponentially small concentration under an exponential prior on the local scales. From \eqref{thm:pm_conc} in Theorem \ref{thm:basic_stmt}, appropriate point mass mixture priors \eqref{eq:point_mass} would have $\bbP( \norm{\theta - \theta_0}_2 < t_n) \geq e^{ - C \norm{\theta_0}_1}$ under the same conditions as above, clearly showing that the wide difference in concentration still persists.


\subsection{\bf{Posterior lower bounds in normal means}}\label{subsec:norm_means}

We have discussed the prior concentration for a high-dimensional vector $\theta$ without alluding to any specific model so far.
In this section we show how prior concentration impacts posterior inference for  the widely studied normal means problem \footnote{Although we study the normal means problem, the ideas and results in this section are applicable to other models such as non-parametric regression and factor models.} (see \cite{donoho1992maximum,johnstone2004needles,castilloneedles} and references therein):  
\begin{align}\label{eq:norm_means}
 y_i &= \theta_i + \epsilon_i, \quad \epsilon_i \sim \mbox{N}(0, 1), \quad 1 \leq i \leq n. 
\end{align}

The minimax rate $s_n$ for the above model is given by $s^2_n  = q_n \log(n/q_n)$ when $\theta_0 \in l_0[q_n;n]$.
For this model \cite{castilloneedles} recently established that  for point mass priors for $\theta$ with $\pi \sim \mbox{beta}(1, \kappa n + 1)$ and $g_{\theta}$ having Laplace like or heavier tails,  the posterior contracts at the minimax rate, \textit{i.e.,}
$\bbE_{n, \theta_0} \bbP(\norm{\theta - \theta_0}_2 < M s_n \mid y) \to 1$ for some constant $M > 0$. Thus we see that carefully chosen point mass priors are indeed optimal\footnote{It is important that the hyper parameter for $\pi$ depends on $n$. We do not know if the result holds without this}. However not all choices for $g_\theta$ lead to optimal proceedures; \cite{castilloneedles} also showed that if $g_{\theta}$ is instead chosen to be standard Gaussian, \textit{the posterior does not contract at the minimax rate}, \textit{i.e.}, one could have $\bbE_{n, \theta_0} \bbP (\norm{\theta - \theta_0}_2 < s_n \mid y) \to 0$ for signals of sufficiently large magnitude. 
This result is particularly striking given the routine choice of Gaussian for $g_{\theta}$ in Bayesian variable selection and thus clearly illustrates the need for careful prior choice in high dimensions.
 
To establish such a posterior lower-bound result, \cite{castilloneedles} showed that given a fixed sequence $t_n$, if there exists a sequence $r_n$ ($r_n > t_n$) such that 
\begin{align}
\frac{\bbP(\norm{\theta - \theta_0}_2 < t_n )}{\bbP(\norm{\theta - \theta_0}_2 < r_n )} = o(e^{-r_n^2}), \label{eq:lb_ratio}
\end{align}
then $\bbP(\norm{\theta - \theta_0}_2 < t_n \mid y) \to 0$. This immediately shows the importance of studying the prior concentration. Intuitively, \eqref{eq:lb_ratio} would be satisfied when the prior mass of the bigger ball $\norm{\theta - \theta_0}_2 < r_n$ is almost entirely contained in the annulus with inner radius $t_n$ and outer radius $r_n$, so that the smaller ball $\norm{\theta - \theta_0}_2 < t_n$ barely has any prior mass compared to the bigger ball. As an illustrative example, in the i.i.d. $\mbox{N}(0, 1)$ example with $t_n = s_n$, setting $r_n = \sqrt{n}$ would satisfy \eqref{eq:lb_ratio} above, proving that i.i.d. $\mbox{N}(0, 1)$ priors are sub-optimal.
Our goal is to investigate whether a similar phenomenon persists for global-local priors in light of the concentration bounds developed in Theorems \ref{thm:conc_g} and \ref{thm:conc_lg}. 

As in Section 3.2, we first state our posterior lower bound result for the case where there is only a global parameter. 
\begin{theorem}\label{thm:lb_g}
Suppose we observe $y \sim \mbox{N}_n(\theta_0, \mathrm{I}_n)$ and \eqref{eq:norm_means} is fitted with a $\mathrm{GL}$ prior on $\theta$ such that $\psi_j = 1$ for all $j$ and the prior $f$ on the global parameter $\tau$ lies in $\mathcal{F}$.  Assume $\theta_0 \in l_0[q_n; n]$ where $q_n/n \to 0$ and $\norm{\theta_0}_2 > s_n$, with $s_n^2 = q_n \log(n / q_n)$ being the minimax squared error loss over $l_0[q_n; n]$. Then, 
$\bbE_{n, \theta_0} \bbP(\norm{\theta - \theta_0}_2 \leq s_n \mid y) \to 0$.
\end{theorem}
\begin{proof}
Without loss of generality, assume $\norm{\theta_0}_2 = o(\sqrt{n})$, since the posterior mass with a prior centered at the origin would be smaller otherwise. Choosing $t_n = s_n$, $r_n$ to be a sequence such that $t_n < r_n < \norm{\theta_0}_2$ and resorting to the two-sided bounds in Theorem \ref{thm:conc_g}, the ratio in \eqref{eq:lb_ratio} is smaller than $(t_n/r_n)^n$, and hence $e^{r_n^2} (t_n/r_n)^n \to 0$ since $r_n \leq \norm{\theta_0}_2 = o(\sqrt{n})$. 
\end{proof}
Theorem \ref{thm:lb_g} states that a GL prior with only a global scale is sub-optimal if $\norm{\theta_0}_2 > s_n$. Observe that in the complementary region $\{ \norm{\theta_0}_2 \leq s_n\}$, the estimator $\hat{\theta} \equiv 0$ attains squared error in the order of $q_n \log(n/q_n)$, implying the condition $\norm{\theta_0}_2 > s_n$ is hardly stringent. 

Next, we state a result for the sub-class of GL priors as in Theorem \ref{thm:conc_lg}, i.e., when $g$ has an exponential distribution leading to a double-exponential distribution marginally.  
\begin{theorem}\label{thm:lb_lg}
Suppose we observe $y \sim \mbox{N}_n(\theta_0, \mathrm{I}_n)$ and the model in \eqref{eq:norm_means} is fitted with a $\mathrm{GL}$ prior on $\theta$ such that $f$ lies in $\mathcal{F}$ and $g \equiv \mathrm{Exp}(\lambda)$ for some constant $\lambda > 0$. Assume $\theta_0 \in l_0[q_n; n]$ with $q_n = 1$ and $\norm{\theta_0}_2^2 /\log n \to \infty$. Then, $\bbE_{n, \theta_0} \bbP(\norm{\theta - \theta_0}_2 \leq \sqrt{\log n} \mid y) \to 0$. 
\end{theorem}
A proof of Theorem \ref{thm:lb_lg} is deferred to Section \ref{sec:main_pf}. From \cite{castilloneedles}, appropriate point mass mixture priors would assign increasing mass with $n$ to the same neighborhood in Theorem \ref{thm:lb_lg}. Hence, many of the shrinkage priors used in practice are sub-optimal in high-dimensional applications, even in the simplest deviation from the null model with only one moderately sized signal. Although Theorem \ref{thm:lb_lg} is stated and proved for $g$ having an exponential distribution (which includes the Bayesian lasso \cite{park2008bayesian}), we conjecture that the conclusions would continue to be valid if one only assumes $g$ to have exponential tails plus some mild conditions on the behavior near zero. However, the assumptions of Theorem \ref{thm:lb_lg} precludes the case when $g$ has polynomial tails, such as the horseshoe \cite{carvalho2010horseshoe} and generalized double Pareto \cite{armagan2011generalized}. One no longer obtains tight bounds on the prior concentration for $g$ having polynomial tails using the current techniques and it becomes substantially complicated to study the posterior. 

Another important question beyond the scope of the current paper should concern the behavior of the posterior when one plugs in an empirical Bayes estimator of the global parameter $\tau$. However, we show below that the ``optimal'' sample-size dependent plug-in choice $\tau_n = c^2/\log n$ (so that marginally $\theta_j \sim \mbox{DE}(c/\sqrt{\log n})$ ) for the lasso estimator \cite{negahban2010unified} produces a sub-optimal posterior:
\begin{theorem}\label{thm:lb_plugin}
Suppose we observe $y \sim \mbox{N}_n(\theta_0, \mathrm{I}_n)$ and \eqref{eq:norm_means} is fitted with a $\mathrm{GL}$ prior on $\theta$ such that $\tau$ is   deterministically chosen to be $\tau_n$, i.e., $f \equiv \delta_{\tau_n}$ for a non-random sequence $\tau_n$ and $g \equiv \mathrm{Exp}(\lambda)$ for some constant $\lambda > 0$. Assume $\theta_0 \in l_0[q_n; n]$ with $q_n (\log n)^2 = o(n)$ and $\tau_n = c/\log n$ is used as the plug-in choice. Then, $\bbE_{n, \theta_0} \bbP(\norm{\theta - \theta_0}_2 \leq s_n \mid y) \to 0$, with $s_n^2 = q_n \log(n / q_n)$ being the minimax squared error loss over $l_0[q_n; n]$. 
\end{theorem}
A proof of Theorem \ref{thm:lb_plugin} can be found in Section \ref{sec:main_pf}. Note that a slightly stronger assumption on the sparsity allows us to completely obviate any condition on $\theta_0$ in this case. Also, the result can be generalized to any $\tau_n$ if $q_n \log n/\tau_n = o(n)$. 

\section{A new class of  shrinkage priors}
  
The results in Section \ref{sec:conc_prop} necessitate the development of a general class of continuous shrinkage priors with improved concentration around sparse vectors.  To that end, let us revisit the global-local specification (\ref{eq:lg}). After integrating out the local scales $\psi_j$'s, \eqref{eq:lg} can be equivalently represented as a global scale mixture of a kernel $\mathcal{K}(\cdot)$,
\begin{eqnarray}\label{eq:scaledkernel}
\theta_j \stackrel{\text{i.i.d.}}{\sim} \mathcal{K}(\cdot \; , \tau), \quad \tau \sim g,
\end{eqnarray}
where $\mathcal{K}(x) = \int \psi^{-1/2} \phi(x/\sqrt{\psi}) g(\psi) d\psi $ is a symmetric unimodal density (or kernel) on $\mathbb{R}$ and $\mathcal{K}(x, \tau) = \tau^{-1/2} \mathcal{K}(x/\sqrt{\tau})$. 
For example, $\psi_j \sim \mbox{Exp}(1/2)$ corresponds to a double exponential kernel $\mathcal{K} \equiv \mbox{DE}(1)$, while $\psi_j \sim \mbox{IG}(1/2,1/2)$ results in a standard Cauchy kernel $\mathcal{K} \equiv \mbox{Ca}(0, 1)$. These traditional choices lead to a kernel which is \emph{bounded} in a neighborhood of zero, and the resulting global-local procedure \eqref{eq:scaledkernel} with a single global parameter $\tau$ doesn't attain the desired concentration around sparse vectors as documented in Theorem \ref{cor:thm_lg}, leading to sub-optimal behavior of the posterior in Theorem \ref{thm:lb_lg}. 

However, if one instead uses a half Cauchy prior $\psi_j^{1/2} \sim \mbox{Ca}_+(0,1)$, then the resulting horseshoe kernel \cite{carvalho2010horseshoe,carvalho2009handling} is unbounded with a singularity at zero. This phenomenon coupled with tail robustness properties leads to excellent empirical performances of the horseshoe. However, the joint distribution of $\theta$ under a horseshoe prior is understudied. One can imagine that it achieves a higher prior
 concentration around sparse vectors compared to common shrinkage priors since the singularity at zero potentially allows most of the entries to be concentrated around zero with the heavy tails ensuring concentration around the relatively small number of signals. However, the polynomial tails of $\psi_j$ present a hindrance in obtaining tight bounds using our techniques. 
 We hope to address the polynomial tails case in details elsewhere, though based on strong empirical performance, we conjecture that the horseshoe leads to the optimal posterior contraction in a much broader domain compared to the Bayesian lasso and other common shrinkage priors. The normal-gamma scale mixtures \cite{griffin2010inference} and the generalized double Pareto prior \cite{armagan2011generalized} follow the same philosophy and should have similar properties. 

The above class of priors rely on obtaining a suitable kernel $\mathcal{K}$ through appropriate normal scale mixtures. In this article, we offer a fundamentally different class of shrinkage priors that alleviate the requirements on the kernel, while having attractive theoretical properties. In particular, our proposed class of kernel-Dirichlet (kD) priors replaces the single global scale $\tau$ in \eqref{eq:scaledkernel} by a vector of scales $(\phi_1\tau, \ldots, \phi_n \tau)$, where $\phi = (\phi_1, \ldots, \phi_n)$ is constrained to lie in the $(n-1)$ dimensional simplex $\mathcal{S}^{n-1}$:
\begin{eqnarray}\label{eq:kd}
\theta_j \mid \phi_j, \tau \sim \mathcal{K}(\cdot \;, \phi_j \tau), \quad (\phi, \tau) \in \mathcal{S}^{n-1} \otimes \mathbb{R}^+ , 
\end{eqnarray}
where $\mathcal{K}$ is any symmetric (about zero) unimodal density that can be represented as scale mixture of normals \cite{west1987scale}. While previous shrinkage priors in the literature obtain marginal behavior similar to the point mass mixture priors \eqref{eq:point_mass}, our construction aims at resembling the \emph{joint distribution} of $\theta$ under a two-component mixture prior. 
Constraining $\phi$ on $\mathcal{S}^{n-1}$ restrains the ``degrees of freedom'' of the $\phi_j$'s, offering better control on the number of dominant entries in $\theta$. In particular, letting $\phi \sim \mbox{Dir}(a, \ldots, a)$ for a suitably chosen $a$ allows (\ref{eq:kd}) to behave like \eqref{eq:point_mass} jointly, forcing a large subset of $(\theta_1, \ldots, \theta_n)$ to be \emph{simultaneously} close to zero with high probability.  

We focus on the Laplace kernel from now on for concreteness, noting that all the results stated below can be generalized to other choices. The corresponding hierarchical prior 
\begin{align}\label{eq:laplace_dir}
\theta_j \sim \mathrm{DE}(\phi_j \tau), \quad \phi \sim \mathrm{Dir}(a, \ldots, a), \quad \tau \sim g
\end{align}
is referred to as a Dirichlet Laplace prior, denoted $\mbox{DL}_{a}(\tau)$. In the following Theorem \ref{thm:kD_conc}, we establish the improved prior concentration of the $DL$ prior. For sake of comparison with the global-local priors in Section 3.3, we assume the same conditions as in Theorem \ref{cor:thm_lg}; a general version can be found in Section \ref{sec:main_pf}. 

\begin{theorem}\label{thm:kD_conc}
Assume $\theta \sim \mbox{DL}_{a}(\tau)$ as in \eqref{eq:laplace_dir} with $a = 1/n$ and $\tau \sim \mathrm{Exp}(\lambda)$ for some $\lambda > 0$. Also assume $\theta_0$ has only one non-zero entry and $\norm{\theta_0}_2^2 =  c \log n$. Also, recall the sequence $t_n$ in \eqref{eqn:tn}. Then, for a constant $C$ depending only on $\delta$ on $\lambda$, 
\begin{align}\label{eq:conc_ld_cor}
P( \norm{\theta - \theta_0} < t_n )  \geq  \exp \{ -C \sqrt{\log n}   \}.
\end{align} 
\end{theorem}
From \eqref{thm:pm_conc} in Theorem \ref{thm:basic_stmt}, appropriate point mass mixtures would attain exactly the same concentration as in \eqref{eq:conc_ld_cor}, showing the huge improvement in concentration compared to global-local priors. This further establishes the role of the dependent scales $\phi$, since in absence of $\phi$, a $\mbox{DE}(\tau)$ prior with $\tau \sim \mbox{Exp}(\lambda)$ would lead to a concentration smaller than $e^{- C \sqrt{n} }$ (see Theorem \ref{cor:thm_lg}).

To further understand the role of $\phi$, we undertake a study of the marginal properties of $\theta_j$ integrating out $\phi_j$. Clearly, the marginal distribution of $\phi_j$ is $\mbox{Beta}(a, (n-1)a)$. Let $\mbox{WG}(\alpha, \beta)$ denote a wrapped gamma distribution with density function
\begin{eqnarray*}
f(x; \alpha, \beta) \propto \abs{x}^{\alpha-1} e^{-\beta \abs{x}}, \quad x\in \mathbb{R} . 
\end{eqnarray*}
The results are summarized in Proposition \ref{propWG} below. 
\begin{proposition}\label{propWG}
If $\theta \mid \phi, \tau \sim \mathrm{DL}_{a}(\tau)$ and $\phi \sim \mathrm{Dir}(a, \ldots, a)$, then the marginal distribution of $\theta_j$ given $\tau$ is unbounded with a singularity at zero for any $a < 1$. Further, in the special case $a = 1/n$, the marginal distribution is a wrapped Gamma distribution $\mathrm{WG}(1/n, \tau^{-1})$. 
\end{proposition}
Thus, marginalizing over $\phi$, we obtain an unbounded kernel $\mathcal{K}$ (similar to the horseshoe). Since the marginal density of $\theta_j \mid \tau $ has a singularity at 0, it assigns a huge mass at zero while retaining exponential tails, which partly explains the improved concentration. A proof of Proposition \ref{propWG} can be found in the appendix.   
    
There is a recent frequentist literature on  including a local penalty specific to each coefficient. The adaptive Lasso \citep{zou2006adaptive,wang2007unified} relies on empirically estimated weights that are plugged in.  \cite{leng2010variable} instead propose to sample the penalty parameters from a posterior, with a sparse point estimate obtained for each draw.  These approaches do not produce a full posterior distribution but focus on sparse point estimates. 
 \subsection{\bf Posterior computation}
 
The proposed class of DL priors leads to straightforward posterior computation via an efficient data augmented Gibbs sampler.  Note that the $\mbox{DL}_a(\tau)$ prior \eqref{eq:laplace_dir} can be equivalently represented as
\begin{align*}
\theta_j \sim \mbox{N}(0, \psi_j \phi_j^2 \tau^2), \, \psi_j \sim \mathrm{Exp}(1/2), \, \phi \sim \mbox{Dir}(a, \ldots, a) .
\end{align*}
In the general $\mbox{DL}_a(\tau)$ setting, we assume a $\mbox{gamma}(\lambda, 1/2)$ prior on $\tau$ with $\lambda = n a$. In the special case when $a = 1/n$, the prior on $\tau$ reduces to an $\mbox{Exp}(1/2)$ prior consistent with the statement of Theorem \ref{thm:kD_conc}. 

We detail the steps in the normal means setting but the algorithm is trivially modified to accommodate normal linear regression, robust regression with heavy tailed residuals, probit models, logistic regression, factor models and other hierarchical Gaussian cases.  To reduce auto-correlation, we rely on marginalization and blocking as much as possible.  Our sampler cycles through (i) $\theta \mid \psi, \phi, \tau, y$, (ii) $\psi \mid \phi, \tau, \theta$, (iii) $\tau \mid \phi, \theta$ and (iv) $\phi \mid \theta$. We use the fact that the joint posterior of $(\psi, \phi, \tau)$ is conditionally independent of $y$ given $\theta$. Steps (ii) - (iv) together gives us a draw from the conditional distribution of $(\psi, \phi, \tau) \mid \theta$, since 
\begin{align*}
[\psi, \phi, \tau \mid \theta] = [\psi \mid \phi, \tau, \theta] [\tau \mid \phi, \theta] [\phi \mid \theta] .
\end{align*}
 
Steps (i) -- (iii) are standard and hence not derived. Step (iv) is non-trivial and we develop an efficient sampling algorithm for jointly sampling $\phi$. Usual one at a time updates of a Dirichlet vector leads to tremendously slow mixing and convergence, and hence the joint update in Theorem \ref{thm:post_comp} is an important feature of our proposed prior. 
\begin{theorem}\label{thm:post_comp}
The joint posterior of $\phi \mid \tau$ has the same distribution as $(T_1/T, \ldots, T_n/T)$, where $T_j$ are independently distributed according to a $\mbox{giG}(a-1, 1, 2 | \theta_j| )$ distribution, and $T = \sum_{j=1}^n T_j$. 
\end{theorem}
\begin{proof}
We first state a result from the theory of normalized random measures (see, for example, (36) in \cite{kruijer2010adaptive}). Suppose $T_1, \ldots, T_n$ are independent random variables with $T_j$ having a density $f_j$ on $(0, \infty)$. Let $\phi_j = T_j/T$ with $T = \sum_{j=1}^n T_j$. Then, the joint density $f$ of $(\phi_1, \ldots, \phi_{n-1})$ supported on the simplex $\mathcal{S}^{n-1}$ has the form
\begin{align}\label{eq:normalized_rep}
f(\phi_1, \ldots, \phi_{n-1}) = \int_{t=0}^{\infty} t^{n-1} \prod_{j=1}^n f_j(\phi_j t) dt ,
\end{align}
where $\phi_n = 1 - \sum_{j=1}^{n-1} \phi_j$. Integrating out $\tau$, the joint posterior of $\phi \mid \theta$ has the form
\begin{align}\label{eq:marg_post_phi}
\pi(\phi_1, \ldots, \phi_{n-1} \mid \theta) \propto \prod_{j=1}^n \bigg[ \phi_j^{a - 1} \frac{1}{\phi_j}  \bigg] 
\int_{\tau = 0}^{\infty} e^{-\tau/2} \tau^{\lambda-n-1} e^{- \sum_{j=1}^n | \theta_j | / (\phi_j \tau) } d \tau. 
\end{align}
Setting $f_j(x) \propto \frac{1}{x^{\delta}} e^{- | \theta_j |/ x } e^{- x/2} $ in \eqref{eq:normalized_rep}, we get
\begin{align}\label{eq:nrep1}
f(\phi_1, \ldots, \phi_{n-1})  = \bigg[ \prod_{j=1}^n \frac{1}{\phi_j^{\delta}} \bigg] \int_{t=0}^{\infty} e^{-t/2} t^{n-1 - n\delta} e^{- \sum_{j=1}^n | \theta_j | / (\phi_j t) } dt.
\end{align}
We aim to equate the expression in \eqref{eq:nrep1} with the expression in \eqref{eq:marg_post_phi}. Comparing the exponent of $\phi_j$ gives us $\delta = 2 - a$. The other requirement $n - 1 - n \delta = \lambda - n - 1$ is also satisfied, since $\lambda = na$. The proof is completed by observing that $f_j$ corresponds to a $\mbox{giG}(a-1, 1, 2 | \theta_j| )$ when $\delta = 2 - a$. \end{proof} 

The summary of each step are finally provided below.
\begin{description}
\item[\textbf{(i)}] To sample $\theta \mid \psi, \phi, \tau, y$, draw $\theta_j$ independently from a $\mbox{N}(\mu_j, \sigma_j^2)$ distribution with 
\begin{align*}
\sigma_j^2 = \{1 + 1/(\psi_j \phi_j^2 \tau^2)\}^{-1}, \quad \mu_j = \{1 + 1/(\psi_j \phi_j^2 \tau^2)\}^{-1} y.
\end{align*}

\item[\textbf{(ii)}]  The conditional posterior of $\psi \mid \phi, \tau, \theta$ can be sampled efficiently in a block by independently sampling $\psi_j \mid \phi, \theta$ from an inverse-Gaussian distribution $\mbox{iG}(\mu_j, \lambda)$ with $\mu_j = \phi_j \tau/|\theta_j|, \lambda = 1$. 


\item[\textbf{(iii)}] Sample the conditional posterior of $\tau \mid \phi, \theta$ from a $\mbox{giG}(\lambda-n, 1, 2 \sum_{j=1}^n | \theta_j |/\phi_j )$ distribution. 

\item[\textbf{(iv)}] To sample $\phi \mid \theta$, draw $T_1, \ldots, T_n$ independently with $T_j \sim \mbox{giG}(a-1, 1, 2 | \theta_j| )$ and set $\phi_j = T_j/T$ with $T = \sum_{j=1}^n T_j$. 
\end{description}

\section{Simulation Study}\label{sec:sims}

Since the concentration results presented here are non-asymptotic in nature, we expect the theoretical findings to be reflected in finite-sample performance. In particular, we aim to study whether the improved concentration of the proposed Dirichlet Laplace ($\mbox{DL}_{1/n}$) priors compared to the Bayesian lasso (BL) translate empirically. 
As illustration, we show the results from a replicated simulation study with various dimensionality $n$ and sparsity level $q_n$. In each setting, we have $100$ replicates of a $n$-dimensional vector $y$ sampled from a $\mbox{N}_{n}(\theta_0, \mathrm{I}_n)$ distribution with $\theta_0$ having $q_n$ non-zero entries which are all set to be a constant $A > 0$. We chose two values of $n$, namely $n = 100, 200$. For each $n$, we let $q_n  = 5, 10, 20 \%$ of $n$ and choose $A = 7, 8$. This results in $12$ simulation settings in total. The simulations were designed to mimic the setting in Section 3 where $\theta_0$ is sparse with a few moderate-sized coefficients. 

\begin{table}[htbp]
\caption{Squared error comparison  over 100 replicates} \small
\begin{flushleft} 
\begin{tabular}{ccccccccccccc} \toprule
 {\bf n} & \multicolumn{6}{c}{{\bf 100}} & \multicolumn{6}{c}{{\bf 200}}  \\
 \cmidrule(lr){2-7} \cmidrule(lr){8-13} \\
 {\bf $\mathbf{q_n}$} & \multicolumn{2}{c}{{\bf 5}} & \multicolumn{2}{c}{{\bf 10}} & 
  \multicolumn{2}{c}{{\bf 20}} & \multicolumn{2}{c}{{\bf 5}} & \multicolumn{2}{c}{{\bf 10}} & 
  \multicolumn{2}{c}{{\bf 20}}\\
  \cmidrule(lr){2-3} \cmidrule(lr){4-5}  \cmidrule(lr){6-7} \cmidrule(lr){8-9} \cmidrule(lr){10-11} \cmidrule(lr){12-13}  \\
  {\bf A} & {\bf 7} &{\bf 8} & {\bf 7} &{\bf 8}  & {\bf 7} & {\bf 8} & {\bf 7} & {\bf 8} & {\bf 7} & {\bf 8} &{\bf 7} & {\bf 8} \\ \hline 
  BL & 33.05 & 33.63 & 49.85 & 50.04  & 68.35 & 68.54 & 64.78 & 69.34 & 99.50 & 103.15 &133.17 & 136.83 \\ 
  $DL_{1/n}$ & 8.20 & 7.19 & 17.29 & 15.35  & 32.00 & 29.40 & 16.07 & 14.28 & 33.00 & 30.80 & 65.53 & 59.61 \\  \hline
  LS & 21.25 & 19.09 & 38.68 &37.25  & 68.97 & 69.05 & 41.82 & 41.18& 75.55 & 75.12 &137.21 & 136.25 \\ 
  EBMed & 13.64 & 12.47 & 29.73 &27.96  & 60.52 & 60.22 & 26.10 & 25.52 & 57.19 & 56.05 &119.41 & 119.35 \\ 
  PM & 12.15 & 10.98 & 25.99 &24.59  & 51.36 & 50.98 &22.99 &22.26 & 49.42 &48.42 &101.54 & 101.62\\ 
  HS & 8.30 & 7.93  & 18.39 &16.27 & 37.25 & 35.18 & 15.80 & 15.09 &35.61 & 33.58 &72.15 & 70.23 \\  \hline
 
\end{tabular}
\end{flushleft}
\label{tab:1}
\end{table}

The squared error loss corresponding to the posterior median averaged across simulation replicates is provided in Table \ref{tab:1}. To offer further grounds for comparison, we have also tabulated the results for Lasso (LS), Empirical Bayes median (EBMed) as in \cite{johnstone2004needles} \footnote{The EBMed procedure was implemented using the package \cite{johnstone2005ebayesthresh}. }, posterior median with a point mass prior (PM) as in \cite{castilloneedles} and the posterior median corresponding to the horseshoe prior \cite{carvalho2010horseshoe}. For the fully Bayesian analysis using point mass mixture priors, we use a complexity prior on the subset-size, $\pi_n(s) \propto \exp \{ -\kappa s \log (2n/s) \}$ with $\kappa = 0.1$ and independent standard Laplace priors for the non-zero entries as in \cite{castilloneedles}. \footnote{Given a draw for $s$, a subset $S$ of size $s$ is drawn uniformly. Set $\theta_j = 0$ for all $j \notin S$ and draw $\theta_j, j \in S$ i.i.d. from standard Laplace. }
\footnote{The beta-bernoulli priors in \eqref{eq:point_mass} induce a similar prior on the subset size.}

Even in this succinct summary of the results, a wide difference between the Bayesian Lasso and the proposed $\mbox{DL}_{1/n}$ is observed in Table \ref{tab:1}, vindicating our theoretical results. The horseshoe performs similarly as the $\mbox{DL}_{1/n}$. The superior performance of the $\mbox{DL}_{1/n}$ prior can be attributed to its strong concentration around the origin. However, in cases where there are several relatively small signals, the $\mbox{DL}_{1/n}$ prior can shrink all of them towards zero. In such settings, depending on the practitioner's utility function, the singularity at zero can be ``softened'' using a $\mbox{DL}_a$ prior for a smaller value of $a$. Based on empirical performance and computational efficiency, we recommend $a = 1/2$ as a robust default choice. The computational gain arises from the fact that in this case, the distribution of $T_j$ in (iv) turns out to be inverse-Gaussian (iG), for which exact samplers are available. 

\begin{table}[htbp]
\caption{Squared error comparison over 100 replicates}
\begin{center}
\begin{tabular}{ccccccc} \toprule
 {\bf n} & \multicolumn{6}{c}{{\bf 1000}}\\
 \cmidrule(lr){1-7} \\
   {\bf A} & {\bf 2} &{\bf3} & {\bf4} &{\bf 5}  & {\bf 6} & {\bf 7}  \\ \hline 
  BL & 299.30 & 385.68 & 424.09 & 450.20  & 474.28 & 493.03\\ 
  HS &306.94 & 353.79  & 270.90 &205.43 & 182.99& 168.83  \\ 
 $ \mbox{DL}_{1/n}$ & 368.45 & 679.17 & 671.34 & 374.01  & 213.66 & 160.14  \\ 
 $ \mbox{DL}_{1/2}$ & 267.83 & 315.70 & 266.80 & 213.23  & 192.98 & 177.20 \\ \hline 
\end{tabular}
\end{center}
\label{tab:2}
\end{table}

For illustration purposes, we choose a simulation setting akin to an example in \cite{carvalho2010horseshoe}, where one has a single observation $y$ from a $n = 1000$ dimensional $\mbox{N}_{n}(\theta_0, \mathrm{I}_{n})$ distribution, with $\theta_0[1:10] = 10, \theta_0[11:100] = A$, and $\theta_0[101:1000] = 0$. We the vary $A$ from $2$ to $7$ and summarize the squared error averaged across $100$ replicates in Table \ref{tab:2}. We only compare the Bayesian shrinkage priors here; the squared error for the posterior median is tabulated. Table \ref{tab:2} clearly illustrates the need for prior elicitation in high dimensions according to the need, shrinking the noise vs. signal detection. 

\section{Proofs of concentration results in Section 3}\label{sec:main_pf}
In this section, we develop non-asymptotic bounds to the prior concentration which are subsequently used to prove the posterior lower bound results. An important tool used throughout is a general version of Anderson's lemma \cite{van2008reproducing}, providing a concentration result for multivariate Gaussian distributions: 
\begin{lemma}\label{lem:anderson}
Suppose $\theta \sim \mbox{N}_n(0, \Sigma)$ with $\Sigma$ p.d. and $\theta_0 \in \mathbb{R}^n$. Let $\norm{\theta_0}_{\mathbb{H}}^2 = \theta_0^{\T} \Sigma^{-1} \theta_0$. Then, for any $t > 0$,
\be
e^{- \frac{1}{2} \; \norm{\theta_0}_{\mathbb{H}}^2}  \bbP(\norm{\theta}_2 \leq t/2) \leq \bbP(\norm{\theta - \theta_0}_2 < t) \leq e^{- \frac{1}{2} \; \norm{\theta_0}_{\mathbb{H}}^2}  \bbP(\norm{\theta}_2 < t).
\ee
\end{lemma}
It is well known that among balls of fixed radius, a zero mean multivariate normal distribution places the maximum mass on the ball centered at the origin. Lemma \ref{lem:anderson} provides a sharp bound on the probability of shifted balls in terms of the centered probability and the size of the shift, measured via the RKHS norm $\norm{\theta_0}_{\mathbb{H}}^2$. 

For GL shrinkage priors of the form (\ref{eq:lg}), given $\psi = (\psi_1, \ldots, \psi_n)^{\T}$ and $\tau$, the elements of $\theta$ are conditionally independent with $\theta \mid \psi, \tau \sim \mbox{N}_n(0, \Sigma)$ with $\Sigma = \mathrm{diag}(\psi_1 \tau, \ldots, \psi_n \tau)$. Hence we can use Lemma \ref{lem:anderson} to obtain
\begin{align}\label{eq:main_bd1}
& e^{- 1/(2 \tau) \sum_{j=1}^n \theta_{0j}^2/ \psi_j }  \, \bbP(\norm{\theta}_2 <  t_n/2 \mid \psi, \tau) \leq \bbP( \norm{\theta - \theta_0}_2 \leq t_n \mid \psi, \tau)  \notag \\
& \leq e^{- 1/(2 \tau) \sum_{j=1}^n \theta_{0j}^2/ \psi_j }  \, \bbP(\norm{\theta}_2 <  t_n \mid \psi, \tau) . 
\end{align}
Letting $X_j = \theta_j^2$, $X_j$'s are conditionally independent given $(\tau, \psi)$ with $X_j$ having a density $f(x \mid \tau, \psi) = D /(\sqrt{\tau \psi_j x}) e^{-x/(2\tau \psi_j) }$ on $(0, \infty)$, where $D = 1/(\sqrt{2 \pi})$. Hence, with $w_n = t_n^2$, 
\begin{align}\label{eq:main_bd2}
\bbP(\norm{\theta}_2 < t_n \mid \psi, \tau) =  D^n \int_{\sum x_j \leq w_n} \prod_{j=1}^n \frac{1}{\sqrt{x_j \tau \psi_j }} e^{-x_j/(2 \tau \psi_j )} dx . 
\end{align}
For sake of brevity, we use $\{ \sum x_j \leq w_n \}$ in \eqref{eq:main_bd2} and all future references to denote the region $\{ x : x_j \geq 0 \, \forall \, j = 1, \ldots, n, \, \sum_{j=1}^n x_j \leq w_n \}$. 
To estimate two-sided bounds for the marginal concentration $\bbP( \norm{\theta - \theta_0}_2 \leq t_n)$, we need to combine \eqref{eq:main_bd1} \& \eqref{eq:main_bd2} and integrate out $\psi$ and $\tau$ carefully. We start by proving Theorem \ref{thm:conc_invgam} \& Theorem \ref{thm:conc_g} where one only needs to integrate out $\tau$. 

\subsection{\bf{Proof of Theorem \ref{thm:conc_invgam}}}

In \eqref{eq:main_bd2}, set $\psi_j = 1$ for all $j$, recall $D  = 1/\sqrt{2 \pi}$ and $w_n = t_n^2$, and integrate over $\tau$ to obtain,
\begin{align}\label{eq:glob_marg_c}
\bbP(\norm{\theta}_2 \leq t_n ) 
= D^n \int_{\tau = 0}^{\infty} f(\tau) \bigg[  \int_{\sum x_j \leq w_n} \prod_{j=1}^n \frac{1}{\sqrt{x_j \tau}} e^{-x_j/(2 \tau)} dx\bigg] d\tau .
\end{align} 
Substituting $f(\tau) = c \tau^{-(1+\alpha)} e^{-\beta/\tau}$ with $c = \beta^{\alpha}/\Gamma(\alpha)$ and using Fubini's theorem to interchange the order of integration between $x$ and $\tau$, \eqref{eq:glob_marg_c} equals
\begin{align}
& c D^n \int_{\sum x_j \leq w_n} \prod_{j = 1}^n \frac{1}{\sqrt{x_j}} \bigg [ \int_{\tau = 0}^{\infty} \tau^{-(1 + n/2+ \alpha) } e^{- \frac{1}{2 \tau}(2 \beta + \sum x_j)} d\tau \bigg] dx  \notag \\
& = c D^n 2^{n/2+\alpha} \Gamma(n/2+\alpha) \int_{\sum x_j \leq w_n} \frac{1}{(2 \beta + \sum x_j)^{n/2 + \alpha} } \prod_{j = 1}^n \frac{1}{\sqrt{x_j}}  \, dx \notag \\
& = c D^n 2^{n/2+\alpha} w_n^{n/2} \Gamma(n/2+\alpha)  \int_{\sum x_j \leq 1} \frac{1}{(2 \beta + w_n \sum x_j)^{n/2 + \alpha} } \prod_{j = 1}^n \frac{1}{\sqrt{x_j}}  \, dx. \label{eq:glob_ig_1}
\end{align}
We now state the Dirichlet integral formula (4.635 in \cite{gradshteyn1980corrected}) to simplify a class of integrals as above over the simplex $\Delta^{n-1}$:
\begin{lemma}\label{lem:dir_formula}
Let $h(\cdot)$ be a Lebesgue integrable function and $\alpha_j > 0, j = 1, \ldots, n$. Then,
\be
\int_{\sum x_j \leq 1} h\big(\sum x_j\big) \prod_{j=1}^n x_j^{\alpha_j - 1} dx_1\ldots dx_n = \frac{\prod_{j=1}^n \Gamma(\alpha_j)}{\Gamma \big(\sum_{j=1}^n \alpha_j\big)} \int_{t=0}^1 h(t) \, t^{( \sum \alpha_j) - 1} dt. 
\ee
\end{lemma}
Lemma \ref{lem:dir_formula} follows simply by noting that the left hand side is $\bbE h(\sum_{j=1}^n X_j) $ up to normalizing constants where $(X_1, \ldots, X_n) \sim \mbox{Diri}(\alpha_1, \ldots, \alpha_n, 1)$, so that $\sum_{j=1}^n X_j \sim \mbox{Beta}(\sum \alpha_j, 1)$. Such probabilistic intuitions will be used later to reduce more complicated integrals over a simplex to a single integral on $(0, 1)$. 

Lemma \ref{lem:dir_formula} with $h(t) = 1/(2 \beta + w_n t)^{n/2 + \alpha}$ applied to \eqref{eq:glob_ig_1} implies
\begin{align}
\bbP(\norm{\theta}_2 \leq t_n )  & = c D^n 2^{n/2+\alpha}  w_n^{n/2} \Gamma(n/2+\alpha)  \frac{\Gamma(1/2)^n}{ \Gamma(n/2)} \int_{t=0}^1 \frac{t^{n/2-1}}{(2 \beta + w_n t)^{n/2 + \alpha}} dt.   \label{eq:glob_ig_2}
\end{align}
Substituting $D = 1/\sqrt{2 \pi}$, bounding $(2 \beta + w_n t)^{n/2 + \alpha} \geq (2 \beta)^{\alpha+1} (2 \beta + w_n t)^{n/2-1}$, and letting $\tilde{w}_n = w_n/(2 \beta)$, \eqref{eq:glob_ig_2} can be bounded above by
\begin{align*}
\frac{\Gamma(n/2+\alpha)}{\Gamma(n/2) \Gamma(\alpha) (2 \beta)^{\alpha +1 }} \tilde{w}_n^{n/2} \int_{t=0}^1 \frac{t^{n/2-1}}{(1 + \tilde{w}_n t)^{n/2-1}} dt  \leq \frac{w_n \Gamma(n/2+\alpha)}{\Gamma(n/2) \Gamma(\alpha) (2 \beta)^{\alpha +1 }}  \bigg(\frac{\tilde{w}_n}{ 1 + \tilde{w}_n} \bigg)^{n/2-1},
\end{align*}
where the second inequality above uses $t/(a+t)$ is an increasing function in $t > 0$ for fixed $a > 0$. By definition, $w_n = n^{\delta}$ for $0 < \delta < 1$ and hence $\frac{w_n \Gamma(n/2+\alpha)}{\Gamma(n/2) \Gamma(\alpha) (2 \beta)^{\alpha +1 }}$ can be bounded above by $c^{C_1 \log n}$. Also, using $(1-x)^{1/x} \leq e$ for all $x > 0$, $ \{ \tilde{w}_n/(1 + \tilde{w}_n) \}^{n/2-1}$ can be bound above by $e^{-  C_2 n/w_n} = e^{- C_2 n^{1 - \delta}}$. Hence the overall bound is $e^{- C n^{1 - \delta}}$ for some appropriate constant $C > 0$.  \qed 

\subsection{\bf{Proof of Theorem \ref{thm:conc_g}}}

We start with the upper bound in \eqref{eq:conc_gonly_c}. The steps are similar as above and hence only a sketch is provided. Bounding $f(\tau) \leq M$ and interchanging order of integrals in \eqref{eq:glob_marg_c}, 
\begin{align}\label{eq:glob_gen_c1}
\bbP(\norm{\theta}_2 \leq t_n) \leq M D^n 2^{n/2-1} \Gamma(n/2-1) w_n \int_{\sum x_j \leq 1} \frac{1}{(\sum x_j)^{n/2-1} } \prod_{j = 1}^n \frac{1}{\sqrt{x_j}}  \, dx .
\end{align}
Invoking Lemma \ref{lem:dir_formula} with $h(t) = (1/t)^{n/2 -1}$ in \eqref{eq:glob_gen_c1}, the upper bound in \eqref{eq:conc_gonly_c} is proved:
\begin{align*}
M D^n 2^{n/2-1} \Gamma(n/2-1) w_n \frac{\Gamma(1/2)^n}{\Gamma(n/2)} \int_{x=0}^1 x^{n/2-1}/x^{n/2-1} dx = (M/2) \frac{w_n}{n/2 - 1} = C_2 n^{-(1-\delta)}.
\end{align*}
We turn towards proving the lower bound to the centered concentration in \eqref{eq:conc_gonly_c}. Recalling that $f(\tau) \geq 1/M$ on $(0, 1)$ for $f \in \mathcal{F}$, and interchanging integrals in \eqref{eq:glob_marg_c}, we have, with $K = 1/M$,
\begin{align}\label{eq:glob_c_lb1}
\bbP(\norm{\theta}_2 \leq t_n) \geq K D^n \int_{\sum x_j \leq w_n} \prod_{j = 1}^n \frac{1}{\sqrt{x_j}} \bigg [ \int_{\tau = 0}^{1} \tau^{-n/2 } e^{- \sum x_j / (2 \tau)} d\tau \bigg] dx.
\end{align}
We state Lemma \ref{lem:incomplete_gamma} to lower bound the inner integral over $\tau$; a proof can be found in the Appendix. Recall $\int_{\tau = 0}^{\infty} \tau^{-n/2} e^{-a_n/(2 \tau)} d\tau = \Gamma(n/2-1) (2/a_n)^{n/2-1}$. Lemma \ref{lem:incomplete_gamma} shows that the same integral over $(0, 1)$ is of the same order when $a_n \precsim n$. 
\begin{lemma}\label{lem:incomplete_gamma}
For a sequence $a_n \leq n/(2e)$, $\int_{\tau = 0}^1 \tau^{-n/2} e^{-a_n/(2\tau)} d\tau \geq (2/a_n)^{n/2-1} \Gamma(n/2-1) \xi_n$, where $\xi_n \uparrow 1$ with $(1 - \xi_n) \leq D/\sqrt{n}$ for some constant $D > 0$. 
\end{lemma}
Clearly $\sum x_j \leq w_n$ and hence we can apply Lemma \ref{lem:incomplete_gamma} in \eqref{eq:glob_c_lb1} to get
\begin{align}\label{eq:glob_c_lb2}
\bbP(\norm{\theta}_2 \leq t_n) \geq K \xi_n D^n 2^{n/2-1} \Gamma(n/2-1) w_n \int_{\sum x_j \leq 1} \frac{1}{(\sum x_j)^{n/2-1} } \prod_{j = 1}^n \frac{1}{\sqrt{x_j}}  \, dx .
\end{align}
The rest of the proof proceeds exactly as in the upper bound case from \eqref{eq:glob_gen_c1} onwards. \qed

Finally, we combine Anderson's inequality \eqref{eq:main_bd1} with \eqref{eq:main_bd2} (with $\psi_j = 1$ for all $j$ in this case) to bound the non-centered concentrations in \eqref{eq:conc_gonly_nc}. For the upper bound, we additionally use $f(\tau) \leq M$ for all $\tau$ to obtain
\begin{align}
& \bbP( \norm{\theta - \theta_0}_2 \leq t_n) \leq  M D^n \int_{\sum x_j \leq w_n} \prod_{j = 1}^n \frac{1}{\sqrt{x_j}} \bigg [ \int_{\tau = 0}^{\infty} \tau^{-n/2} e^{-[\norm{\theta_0}_2^2 + \sum x_j  ]/(2 \tau)} d\tau \bigg] dx \label{eq:glob_ub_nc1} \\
& = M D^n 2^{n/2-1} \Gamma(n/2-1) w_n^{n/2} \int_{\sum x_j \leq 1} \frac{1}{(\norm{\theta_0}_2^2 + w_n \sum x_j)^{n/2-1} } \prod_{j = 1}^n \frac{1}{\sqrt{x_j}}  \, dx \label{eq:glob_ub_nc2} \\ 
& = M D^n 2^{n/2-1} \Gamma(n/2-1) w_n^{n/2} \frac{\Gamma(1/2)^n}{\Gamma(n/2)} \int_{x=0}^1 \frac{x^{n/2 - 1} }{(\norm{\theta_0}_2^2 +  w_n x)^{n/2 - 1} }.  \label{eq:glob_ub_nc3}
\end{align} 
In the above display, \eqref{eq:glob_ub_nc2} - \eqref{eq:glob_ub_nc3} follows from applying Lemma \ref{lem:dir_formula}  with $h(t) = 1/(\norm{\theta_0}_2^2 +  w_n t)^{n/2-1}$. Simplifying constants in \eqref{eq:glob_ub_nc3} as before and using $t/(a+t)$ is an increasing function in $t > 0$ for fixed $a > 0$, we complete the proof by bounding \eqref{eq:glob_ub_nc3} above by
\begin{align*}
\frac{C w_n}{(n/2-1)} \int_{x=0}^1 \frac{(w_n x)^{n/2 - 1} }{(\norm{\theta_0}_2^2 +  w_n x)^{n/2 - 1} }dx \leq  \frac{C w_n}{(n/2-1)} \bigg(\frac{w_n}{w_n + \norm{\theta_0}_2^2} \bigg)^{n/2-1} \leq \frac{C w_n}{(n/2-1)}  \bigg(\frac{w_n}{\norm{\theta_0}_2^2} \bigg)^{n/2-1}. 
\end{align*}
The right hand side of the above display can be bounded above by $e^{- c n \log a_n}$ for some constant $c > 0$.
Remark \eqref{rem:thm:conc_g} readily follows from the above display; we didn't use the condition on $\norm{\theta_0}_2$ so far. 

For the lower bound on the prior concentration in the non-centered case, we combine Anderson's inequality \eqref{eq:main_bd1} in the reverse direction along with \eqref{eq:main_bd2}. We then use the same trick as in the centered case to restrict the integral over $\tau$ to $(0, 1)$ in \eqref{eq:glob_lb_nc1}. Note that the integral over the $x$'s is over $\sum x_j \leq v_n$ with $v_n = t_n^2/4$ as a consequence of \eqref{eq:main_bd1}. Hence, 
\begin{align}
& \bbP( \norm{\theta - \theta_0}_2 \leq t_n) \geq  K D^n \int_{\sum x_j \leq v_n} \prod_{j = 1}^n \frac{1}{\sqrt{x_j}} \bigg [ \int_{\tau = 0}^{1} \tau^{-n/2} e^{-[\norm{\theta_0}_2^2 + \sum x_j  ]/(2 \tau)} d\tau \bigg] dx. \label{eq:glob_lb_nc1}
\end{align}
Noting that $\norm{\theta_0}_2^2 + \sum x_j \leq \norm{\theta_0}_2^2 + v_n = o(n)$, we can invoke Lemma \ref{lem:incomplete_gamma} to lower bound the inner integral over $\tau$ by $\xi_n \Gamma(n/2-1) 2^{n/2-1}/(\norm{\theta_0}_2^2 + \sum x_j)^{n/2-1} $ and proceed to obtain the same expressions as in \eqref{eq:glob_ub_nc2} \& \eqref{eq:glob_ub_nc3} with $M$ replaced by $K \xi_n$ and $w_n$ by $v_n$. The proof is then completed by observing that the resulting lower bound can be further bounded below as follows:
\begin{align*} 
& \frac{C v_n}{(n/2-1)} \int_{x=0}^1 \frac{(v_n x)^{n/2 - 1} }{(\norm{\theta_0}_2^2 +  v_n x)^{n/2 - 1} }dx \geq \frac{C v_n}{(n/2-1)} \int_{x=1/2}^1 \frac{(v_n x)^{n/2 - 1} }{(\norm{\theta_0}_2^2 +  v_n x)^{n/2 - 1} }dx  \\
&  \geq \frac{C v_n}{(n/2-1)}  \bigg(\frac{v_n/2 }{(\norm{\theta_0}_2^2 +  v_n/2) }\bigg)^{n/2-1} \geq \frac{C v_n}{(n/2-1)} \bigg(\frac{v_n/2 }{2 \norm{\theta_0}_2^2}\bigg)^{n/2 - 1},
\end{align*}
where the last inequality uses $t_n \leq \norm{\theta_0}_2$ so that $\norm{\theta_0}_2^2 + v_n \leq 2 \norm{\theta_0}_2^2$. \qed

To prove Theorem \ref{cor:thm_lg}, we state and prove a more general result on concentration of GL priors. 
\begin{theorem}\label{thm:conc_lg}
Assume $\theta \sim \mathrm{GL}$ with $f \in \mathcal{F}$ and $g \equiv \mathrm{Exp}(\lambda)$ for some constant $\lambda > 0$. Also assume $\theta_0$ has only one non-zero entry. Let $w_n = t_n^2$. Then, for a global constant $C_1 > 0$ depending only on $M$ in the definition of $\mathcal{F}$,
 \begin{align}\label{eq:lg_ub_stmt}
& \bbP(\norm{ \theta - \theta_0}_2 \leq t_n) \leq C_1 \, \int_{\psi_1 = 0}^{\infty} \frac{\psi_1^{(n-3)/2}}{ \big\{ \psi_1 + \norm{\theta_0}_2^2/(\pi w_n) \big \}^{(n-3)/2}  } e^{-\psi_1} d \psi_1 . 
\end{align}
Let $v_n = r_n^2/4$ satisfy $v_n = O(\sqrt{n})$. Then, for $\norm{\theta_0}_2 \geq 1/\sqrt{n}$, 
\begin{align}\label{eq:lg_lb_stmt}
& \bbP(\norm{ \theta - \theta_0}_2 \leq r_n) \geq C_2 e^{-d_2 \sqrt{n}} \, \int_{\psi_1 =c_1 \norm{\theta_0}_2^2}^{\infty} \frac{\psi_1^{(n-3)/2}}{ \big\{ \psi_1 + \norm{\theta_0}_2^2/(\pi v_n) \big \}^{(n-3)/2}  } e^{-\psi_1} d \psi_1,
\end{align}  
where $c_1, d_2, C_2$ are positive global constants with $c_1 \geq 2$ and $C_2$ depends only on $M$ in the definition of $\mathcal{F}$. 
\end{theorem}
\subsection{\bf{Proof of Theorem \ref{thm:conc_lg}}}
Without loss of generality, we assume $g$ to be the $\mbox{Exp}(1)$ distribution since the rate parameter $\lambda$ can be absorbed into the global parameter $\tau$ with the resulting distribution still in $\mathcal{F}$. Also, assume the only non-zero entry in $\theta_0$ is $\theta_{01}$, so that $\norm{\theta_0}_2^2 = \abs{\theta_{01}}^2$. The steps of the proof follow the same structure as in Theorem \ref{thm:conc_g}, i.e., using Anderson's inequality to bound the non-centered concentration given $\psi, \tau$ by the centered concentration as in \eqref{eq:main_bd1} and exploiting the properties of $\mathcal{F}$ to ensure that the bounds are tight. A substantial additional complication arises in integrating out $\psi$ in this case, requiring involved analysis. 

We start with the upper bound \eqref{eq:lg_ub_stmt}. Combining \eqref{eq:main_bd1} \& \eqref{eq:main_bd2}, and bounding $f(\tau) \leq M$ yields:
\be
& \bbP( \norm{\theta - \theta_0}_2 \leq t_n) \\
& \leq  D^n \int_{\tau = 0}^{\infty} f(\tau) e^{- 1/(2 \tau) \sum_{j=1}^n \theta_{0j}^2/ \psi_j } \int_{\psi} g(\psi) \bigg[  \int_{\sum x_j \leq w_n} \prod_{j=1}^n \frac{1}{\sqrt{x_j \tau \psi_j }} e^{-x_j/(2 \tau \psi_j )} dx\bigg] d \psi d\tau  \\
& \leq M D^n \int_{\sum x_j \leq w_n} \prod_{j=1}^n \frac{1}{ \sqrt{x_j} } \int_{\psi} \prod_{j=1}^n \frac{g(\psi_j)}{ \sqrt{\psi_j} } \bigg [ 
\int_{\tau = 0}^{\infty} \tau^{-n/2} e^{- \frac{1}{2 \tau}  \big[\theta_{01}^2/\psi_1  + \sum x_j/\psi_j \big] } d\tau \bigg ] d\psi dx \\
& = M D^n 2^{n/2-1} \Gamma(n/2-1) w_n^{n/2}  \int_{\psi} \prod_{j=1}^n \frac{g(\psi_j)}{ \sqrt{\psi_j} } 
\bigg [\int_{\sum x_j \leq 1 } \frac{\prod_{j=1}^n x_j^{-1/2} }{  [ \norm{\theta_0}_2^2/\psi_1  + w_n \sum x_j/\psi_j ]^{n/2-1} }   \; dx \bigg ] d \psi. \\ \label{eq:lg_ub1}
\ee
Comare \eqref{eq:lg_ub1} with \eqref{eq:glob_ub_nc2}. The crucial difference in this case is that the inner integral over the simplex $\sum_{j=1}^n x_j \leq 1$ is no longer a function of $\sum_{j=1}^n x_j$, rendering Lemma \ref{lem:dir_formula} inapplicable. An important technical contribution of this paper in Lemma \ref{lem:dickey_integral} below is that complicated multiple integrals over the simplex as above can be reduced to a single integral over $(0, 1)$:
\begin{lemma}\label{lem:dickey_integral}
Let $\alpha_j  = 1/2$ for $j = 1, \ldots, n$ and $q_j, j = 0, 1, \ldots, n$ be positive numbers. Then, 
\begin{align*}
\int_{\sum x_j \leq 1 } \frac{\prod_{j=1}^n x_j^{\alpha_j - 1}}{  [ \sum_{j=1}^n q_j x_j + q_0 ]^{n/2-1} } dx = \frac{\Gamma(1/2)^n}{\Gamma(n/2)} q_0(n/2-1) \int_{x=0}^1 \frac{x^{n/2-2} (1-x)}{\prod_{j=1}^n (q_j x + q_0)^{\alpha_j}} dx .
\end{align*}
\end{lemma}
A proof of Lemma \ref{lem:dickey_integral} can be found in the Appendix. We didn't find any previous instance of Lemma \ref{lem:dickey_integral} though a related integral with $n/2$ in the exponent in the denominator appears in \cite{gradshteyn1980corrected}. Our technique for the proof, which utilizes a beautiful identity found in \cite{dickey1968three} can be easily generalized to any $\alpha_j$ and other exponents in the denominator.

Aplying Lemma \ref{lem:dickey_integral} with $q_0 = \norm{\theta_0}_2^2/\psi_1$ and $q_j = w_n/\psi_j$ to evaluate the inner integral over $x$, \eqref{eq:lg_ub1} equals
\begin{align}
(M \norm{\theta_0}_2^2/2) w_n^{n/2} \int_{\psi} \bigg[ \prod_{j=1}^n \frac{g(\psi_j)}{ \sqrt{\psi_j} } \bigg] \frac{1}{ \psi_1} \int_{x=0}^1 \frac{x^{n/2-2} (1-x)}{\prod_{j=1}^n \sqrt{ (w_n x/\psi_j + q_0)} } dx d\psi, \label{eq:lg_ub3}
\end{align}
noting that $(n/2-1) D^n 2^{n/2-1} \Gamma(n/2-1) \, \Gamma(1/2)^n/\Gamma(n/2) = 1/2$. 

So, at this point, we are down from the initial $(2n+1)$ integrals to $(n+1)$ integrals. Next, using $g(\psi_j) = e^{-\psi_j} 1(\psi_j > 0)$ to integrate out $\psi_j, j = 2, \ldots, n$, \eqref{eq:lg_ub3} equals
\begin{align}
(M  \norm{\theta_0}_2^2/2) w_n^{n/2} \int_{\psi_1 = 0}^{\infty} \frac{e^{-\psi_1}}{\psi_1 \sqrt{\psi_1}}  \int_{x=0}^1 \frac{x^{n/2-2} (1-x)}{\sqrt{w_n x/\psi_1 + q_0}} \bigg\{ \int_{\psi = 0}^{\infty} \frac{e^{-\psi}}{ \sqrt{w_n x + \psi q_0} } d \psi \bigg\}^{n-1}dx \, d\psi_1. \label{eq:lg_ub3a}
\end{align}
Using a standard identity and an upper bound for the complementary error function $\mbox{erfc}(z) =  2/\sqrt{\pi} \int_{t=z}^{\infty} e^{-t^2} dt$ (see \ref{eq:erfc_ub} in the Appendix),
\begin{align*}
\int_{\psi = 0}^{\infty} \frac{e^{-\psi}}{ \sqrt{w_n x + \psi q_0} } d \psi = \frac{\sqrt{\pi}}{\sqrt{q_0}} \exp(w_n x/q_0) \mbox{erfc}(\sqrt{w_n x/q_0}) \leq \frac{1}{\sqrt{w_n x + q_0/\pi} } . 
\end{align*} 
Hence, the expression in \eqref{eq:lg_ub3a} can be bounded above by
\begin{align}
& (M/2) \norm{\theta_0}_2^2 w_n^{n/2} \int_{\psi_1 = 0}^{\infty} \frac{e^{-\psi_1}}{ \psi_1} \int_{x=0}^{1} \frac{ x^{n/2-2} (1-x) }{ \sqrt{ \big( w_n x +  \norm{\theta_0}_2^2  \big)} \, \, \big[w_n x +  \norm{\theta_0}_2^2 /(\pi \psi_1) \big] ^{(n-1)/2} } dx \, d\psi_1 \notag \\
& = (M/2) \norm{\theta_0}_2^2 w_n^{n/2}  \int_{\psi_1 = 0}^{\infty} e^{-\psi_1} \psi_1^{(n-3)/2} \int_{x=0}^{1} \frac{ x^{n/2 - 2} (1-x) }{ \sqrt{ \big( w_n x + \norm{\theta_0}_2^2 \big)} \, \, \big[w_n x \psi_1 + \norm{\theta_0}_2^2/\pi \big]^{(n-1)/2} } dx \, d\psi_1. \label{eq:lg_ub4}
\end{align}
Let us aim to bound the inner integral over $x$ in \eqref{eq:lg_ub4}. We upper bound $(1-x)$ in the numerator by $1$, lower-bound $\sqrt{(w_n x + \norm{\theta_0}_2^2)}$ in the denominator by $\sqrt{\norm{\theta_0}_2^2}$ and multiply a $\sqrt{w_n x \psi_1 + \norm{\theta_0}_2^2/\pi}$ term in the numerator and denominator to get 
\begin{align*}
& \int_{x=0}^{1} \frac{ x^{n/2 - 2} (1-x) }{ \sqrt{w_n x + \norm{\theta_0}_2^2} \, \, \big[w_n x \psi_1 + \norm{\theta_0}_2^2/\pi \big]^{(n-1)/2} } dx \\
& \leq \frac{ \sqrt{w_n \psi_1 + \norm{\theta_0}_2^2/\pi} }{\sqrt{\norm{\theta_0}_2^2} } \int_{x=0}^{1} \frac{ x^{n/2-2} }{ \big(w_n x \psi_1 + \norm{\theta_0}_2^2/\pi \big)^{n/2} } dx.
\end{align*}
We use the fact that $\int_{x=0}^{1} x^{n/2-2}/(\alpha x + \beta)^{n/2} dx = 2(\alpha + \beta)^{1 - n/2}/\{ \beta (n-2)\}$ to conclude that the last line in the above display equals
\begin{align*}
& \frac{ \sqrt{w_n \psi_1 + \norm{\theta_0}_2^2/\pi} }{\sqrt{\norm{\theta_0}_2^2} }  \, \frac{2 \pi }{ \norm{\theta_0}_2^2 } \, \frac{ \big(w_n \psi_1 + \norm{\theta_0}_2^2/\pi \big)^{1 - n/2} }{ (n-2) } \\
& = \frac{ 1 }{\sqrt{\norm{\theta_0}_2^2} } \, \frac{\pi}{\norm{\theta_0}_2^2  (n/2-1)} \big(w_n \psi_1 + \norm{\theta_0}_2^2/\pi \big)^{- (n-3)/2}. 
\end{align*}
Substituitng this in \eqref{eq:lg_ub4}, we finally obtain:
\begin{align}
\bbP(\norm{ \theta - \theta_0}_2 \leq t_n) \leq \frac{C_1 w_n }{(n/2-1)}  \, \sqrt{\frac{w_n}{ \norm{\theta_0}_2^2}} \, \int_{\psi_1 = 0}^{\infty} \frac{\psi_1^{(n-3)/2}}{ \big\{ \psi_1 + \norm{\theta_0}_2^2/(\pi w_n) \big \}^{(n-3)/2}  } e^{-\psi_1} d \psi_1, \label{eq:lg_ub_final}
\end{align}
where $C_1 > 0$ is a global constant (depending only on $M$).  \eqref{eq:lg_ub_stmt} clearly follows from \eqref{eq:lg_ub_final}. \qed

{\bf Lower bound:} We proceed to obtain a lower bound to $\bbP(\norm{ \theta - \theta_0}_2 < r_n)$ similar to \eqref{eq:lg_ub_final} under additional assumptions on $r_n$ as in the statement of Theorem \ref{thm:conc_lg}. To that end, note that in the proof of the upper bound here, we used only two inequalities until \eqref{eq:lg_ub3}: (i) Anderson's inequality in \eqref{eq:main_bd1} and (ii) upper bounding $f(\tau)$ by $M$. As in the proof of the lower bound in Theorem \ref{thm:conc_g}, we obtain a lower bound similar to the expression in \eqref{eq:lg_ub3} by (i) using Anderson's inequality \eqref{eq:main_bd1} in the reverse direction, and (ii) using $f(\tau) \geq K$ on $(0, 1)$:
\be
& \bbP( \norm{\theta - \theta_0}_2 \leq r_n) \\
& \geq K D^n \int_{\sum x_j \leq v_n} \prod_{j=1}^n \frac{1}{ \sqrt{x_j} } \int_{\psi} \prod_{j=1}^n \frac{g(\psi_j)}{ \sqrt{\psi_j} } \bigg [ 
\int_{\tau = 0}^{1} \tau^{-n/2} e^{- \frac{1}{2 \tau}  \big[\norm{\theta_0}_2^2/\psi_1 + \sum_{j=1}^n x_j/\psi_j)  \big] } d\tau \bigg ] d\psi \, dx.  \label{eq:lg_lb1}
\ee
However, unlike Theorem \ref{thm:conc_g}, we cannot directly resort to Lemma \ref{lem:incomplete_gamma} since $a_n = \norm{\theta_0}_2^2/\psi_1 + \sum_{j=1}^n x_j/\psi_j$ can be arbitrarily large if $\psi_j$'s are close enough to zero. This necessitates a more careful analysis in bounding below the expression in \eqref{eq:lg_lb1} by constraining the $\psi_j$'s to an appropriate region $\Gamma$ away from zero:
\begin{align*}
\Gamma = \bigg\{ c_1 \norm{\theta_0}_2^2 \leq \psi_1 \leq c_2\norm{\theta_0}_2^2, \, \psi_j \geq c_3/\sqrt{n}, \, j = 2, \ldots, n \bigg\}. 
\end{align*}
In the above display, $c_1 < c_2$ and $c_3 > 1$ are positive constants to be chosen later, that satisfy
\begin{align}\label{eq:lg_lb_cond1}
1/c_1 + \max\{1/(c_1 \norm{\theta_0}_2^2), \sqrt{n}/c_3\} v_n \leq n/(2e) . 
\end{align} 
With \eqref{eq:lg_lb_cond1}, we can invoke Lemma \ref{lem:incomplete_gamma} to bound below the integral over $\tau$ in \eqref{eq:lg_lb1}, since for $\psi \in \Gamma$, $\norm{\theta_0}_2^2/\psi_1 + \sum_{j=1}^n x_j/\psi_j \leq 1/c_1 + \max\{1/(c_1 \norm{\theta_0}_2^2), \sqrt{n}/c_3 \} \sum_{j=1}^n x_j \leq 1/c_1 + \max\{1/(c_1 \norm{\theta_0}_2^2), \sqrt{n}/c_3\} v_n \leq n/(2e)$ by \eqref{eq:lg_lb_cond1}. The resulting lower bound is exactly same as \eqref{eq:lg_ub1} with $M$ replaced by $K \xi_n$ and $w_n$ by $v_n$, where $\xi_n \uparrow 1$ is as in Lemma \ref{lem:incomplete_gamma}. As in the upper bound calculations \eqref{eq:lg_ub1} - \eqref{eq:lg_ub3}, we invoke Lemma \ref{lem:dickey_integral} with $q_0 = \norm{\theta_0}_2^2/\psi_1$ and $q_j = v_n/\psi_j$ to reduce the multiple integral over the simplex and bound the expression in \eqref{eq:lg_lb1} below by
\begin{align}
& (K \norm{\theta_0}_2^2/2) \xi_n v_n^{n/2} \int_{\psi \in \Gamma} \bigg[ \prod_{j=1}^n \frac{g(\psi_j)}{ \sqrt{\psi_j} } \bigg] \frac{1}{ \psi_1} \int_{x=1/2}^{3/4} \frac{x^{n/2-2} (1-x)}{\prod_{j=1}^n \sqrt{ (v_n x/\psi_j + q_0)} } dx d\psi  = \notag \\
& (K \norm{\theta_0}_2^2/2) \xi_n v_n^{n/2}  \int_{\psi_1 = c_1 \norm{\theta_0}^2}^{c_2 \norm{\theta_0}^2} \frac{e^{-\psi_1}}{\psi_1}  \int_{x=1/2}^{3/4}\frac{x^{n/2-2} (1-x)}{\sqrt{v_n x + q_0 \psi_1}} \bigg\{ \int_{\psi = c_3/\sqrt{n}}^{\infty} \frac{e^{-\psi}}{ \sqrt{v_n x + \psi q_0} } d \psi \bigg\}^{n-1}dx \, d\psi_1. \label{eq:lg_lb2}
\end{align}
Note the inner integral over $x$ is restricted to $(1/2, 3/4)$. Now, 
\begin{align}
\int_{\psi = c_3/\sqrt{n}}^{\infty} \frac{e^{-\psi}}{ \sqrt{v_n x + \psi q_0} } d \psi = \frac{ \sqrt{\pi} }{ \sqrt{q_0} } e^{v_n x/q_0} \mbox{erfc}\bigg ( \sqrt{v_nx/q_0 + c_3/\sqrt{n} } \bigg). \label{eq:lg_lb3}
\end{align}
We use a lower bound on the $\mbox{erfc}$ function (see \ref{eq:erfc_lb} in the Appendix for a proof) which states that for $z \geq 2$, $\sqrt{\pi} e^{z} \mbox{erfc}(\sqrt{z}) \geq \big(1/\sqrt{z}\big)^{1+\delta}$ for any $\delta > 0$. Since we have restricted $x \geq 1/2$ in \eqref{eq:lg_lb2} and $v_n \psi_1/\norm{\theta_0}_2^2 \geq c_1v_n $ on $\Gamma$, we have $\sqrt{v_nx/q_0 + c_3/\sqrt{n} }\geq \sqrt{c_1}$ provided $v_n \geq 1$. Thus, choosing $c_1 >  2$, we can apply the above lower bound on the error function to bound the expression in the r.h.s. of \eqref{eq:lg_lb3} as:
\begin{align*}
& \frac{ \sqrt{\pi} }{ \sqrt{q_0} } e^{v_n x/q_0} \mbox{erfc}\bigg ( \sqrt{v_nx/q_0 + c_3/\sqrt{n} } \bigg)  
\geq \frac{1}{\sqrt{q_0}} e^{-c_3/\sqrt{n}}  \bigg [\frac{1}{\sqrt{v_n x/q_0 + c_3/\sqrt{n}}} \bigg]^{1+ \delta} \\
& \geq \frac{1}{\sqrt{q_0}} e^{-c_3/\sqrt{n}}  \frac{1}{\sqrt{v_n x/q_0 + 3/(4\pi)}}   \frac{1}{(1+c_2)^{\delta}} = \frac{e^{-c_3/\sqrt{n}}}{(1 + c_2)^{\delta}} \frac{1}{\sqrt{v_n x + 3q_0/(4\pi)}}. 
\end{align*}
In the second to third step, we used that $v_n x/q_0 + c_3/\sqrt{n} \leq v_n x/q_0 + 3/(4\pi)$ for $n$ larger than some constant. We choose $\delta = 1/(n-1)$ and substitute the above lower bound for the l.h.s. of \eqref{eq:lg_lb3} into \eqref{eq:lg_lb2}. This allows us to bound \eqref{eq:lg_lb2} below by
\begin{align}
& C_1 \xi_n \norm{\theta_0}_2^2 e^{-c_3 (n-1)/\sqrt{n}} v_n^{n/2} \times   \notag \\
& \int_{\psi_1 = c_1 \norm{\theta_0}^2}^{c_2 \norm{\theta_0}^2} e^{-\psi_1} \psi_1^{(n-3)/2} \int_{x=1/2}^{3/4} \frac{ x^{n/2 - 2} (1-x) }{ \sqrt{ \big( v_n x + \norm{\theta_0}_2^2 \big)} \, \, \big[v_n \psi_1 x + 3 \norm{\theta_0}_2^2/(4\pi) \big]^{(n-1)/2} } dx \, d\psi_1 .\label{eq:lg_lb4}
\end{align} 
Let us tackle the integral over $x$ in \eqref{eq:lg_lb4}. To that end, we first lower-bound $(1 - x)$ in the numerator by $1/4$, upper-bound $\sqrt{v_n x + \norm{\theta_0}_2^2}$ in the denominator by $\sqrt{v_n + \norm{\theta_0}_2^2}$. Next, we use the formula
\begin{align*}
\int_{x= 1/2}^{3/4} \frac{x^{n/2- 2}}{ (\alpha x + \beta)^{n/2} } dx = \frac{2(\alpha + 4\beta/3)^{1 - n/2}}{\beta(n-2)} \bigg[ 1 - \bigg\{ \frac{\alpha + 4\beta/3}{\alpha + 2\beta}\bigg\}^{n/2-1} \bigg],
\end{align*}
with $\alpha = v_n \psi_1$ and $\beta = 3 \norm{\theta_0}_2^2/(4 \pi)$. Now, $(\alpha + 4\beta/3)/(\alpha + 2 \beta) = 1 - 2\beta/\{ 3(\alpha + 2 \beta) \}$ is bounded away from $0$ and $1$ since $c_1 \norm{\theta_0}_2^2 \leq \alpha \leq c_2 \norm{\theta_0}_2^2$. Thus, 
\begin{align*}
\bigg[ 1 - \bigg\{ \frac{\alpha + 4\beta/3}{\alpha + 2\beta}\bigg\}^{n/2-1} \bigg] \geq 1/2
\end{align*}
for $n$ large. Substituting all these in \eqref{eq:lg_lb4}, we finally obtain:
\begin{align}
\bbP(\norm{ \theta - \theta_0}_2 \leq r_n ) \geq \frac{C_2 \xi_n v_n \exp(-c_3 \sqrt{n})}{(n/2-1)}   \, \sqrt{\frac{v_n}{ v_n + \norm{\theta_0}_2^2}} \, \int_{\psi_1 = c_1 \norm{\theta_0}^2}^{c_2 \norm{\theta_0}^2} \frac{\psi_1^{(n-3)/2}}{ \big\{ \psi_1 + \norm{\theta_0}_2^2/(\pi v_n) \big \}^{(n-3)/2}  } e^{-\psi_1} d \psi_1, \label{eq:lg_lb_final}
\end{align}
where $C_2 > 0$ is a global constant depending only on $K$ in the definition of $\mathcal{F}$ and $\xi_n \uparrow 1$ with $1 - \xi_n \leq D/\sqrt{n}$ for some constant $D > 0$. We only required $c_1 > 2$ so far. Since $\norm{\theta_0}_2 \geq 1/\sqrt{n}$, choosing $c_1$ and $c_3$ to be sufficiently large constants, \eqref{eq:lg_lb_cond1} can always be satisfied. The proof of \eqref{eq:lg_lb_stmt} clearly follows from \eqref{eq:lg_lb_final}, since $\xi_n v_n/(n/2-1) \sqrt{\frac{v_n}{ v_n + \norm{\theta_0}_2^2}}$ can be bounded below by $e^{- c_4 \sqrt{n} }$. 
\qed

\subsection{\bf Proof of Theorem \ref{thm:lb_lg}}
Let $m_n = (n-3)/2$. We set $t_n = s_n$, where $s_n$ is the minimax rate corresponding to $q_n = 1$, so that $w_n = s_n^2 = \log n$. Also, let $\norm{\theta_0}_2^2 = \pi w_n u_n^2$, where $u_n$ is a slowly increasing sequence; we set $u_n = \log (\log n)$ for future references. Finally let $v_n = r_n^2/4 = \sqrt{m_n}$. With these choices, we proceed to show that \eqref{eq:lb_ratio} holds. 

We first simplify  \eqref{eq:lg_ub_final} further. The function $x \to x/x(x+a)$ monotonically increases from $0$ to $1$ for any $a > 0$. Thus, for any $T_n > 0$, 
\begin{align}
&\int_{\psi_1 = 0}^{\infty} \frac{\psi_1^{m_n}}{ \big\{ \psi_1 + \norm{\theta_0}_2^2/(\pi w_n) \big \}^{m_n}  } e^{-\psi_1} d \psi_1 \notag \\
& \leq \int_{\psi_1 = 0}^{T_n} \frac{\psi_1^{m_n}}{ \big\{ \psi_1 + \norm{\theta_0}_2^2/(\pi w_n) \big \}^{m_n}  } e^{-\psi_1} d \psi_1 + \int_{\psi_1 = T_n}^{\infty} e^{-\psi_1} d \psi_1  \leq \bigg( \frac{T_n}{T_n + u_n^2} \bigg)^{m_n} + e^{-T_n}. \label{eq:comb1}
\end{align}
We choose an appropriate $T_n$ which gives us the necessary bound, namely $T_n = u_n \sqrt{m_n} $. Then, using the fact that $(1 - x)^{1/x} \leq e^{-1}$ for all $x \in (0, 1)$, we have
\begin{align*}
\bigg( \frac{T_n}{T_n + u_n^2} \bigg)^{m_n}  = \bigg( \frac{\sqrt{m_n}}{\sqrt{m_n} + u_n} \bigg)^{m_n}  = \bigg( 1 - \frac{u_n}{\sqrt{m_n} + u_n} \bigg)^{m_n} \leq e^{- m_n u_n/(\sqrt{m_n} + u_n)} \leq e^{- u_n \sqrt{m_n} /2},
\end{align*}
where for the last part used that $e^{-1/x}$ is an increasing function and $\sqrt{m_n} + u_n \leq 2 \sqrt{m_n}$. Thus, substituting $T_n$ in \eqref{eq:comb1} yields, for a global constant $C_1 > 0$, 
\begin{align}\label{eq:lg_simp_ub}
\bbP(\norm{\theta - \theta_0}_2 \leq s_n) \leq \frac{C_1 w_n }{(n/2-1)}  \, \sqrt{\frac{w_n}{ \norm{\theta_0}_2^2}} \,  \, e^{-u_n \sqrt{m_n}/2} .
\end{align}
Next, again using the fact that $x \to x/x(x+a)$ is monotonically increasing, and choosing $c_2 = \infty$, we simplify the lower bound \eqref{eq:lg_lb_final}. Observe
\begin{align*}
& \int_{\psi_1 = c_1 \norm{\theta_0}^2}^{\infty} \frac{\psi_1^{m_n}}{ \big\{ \psi_1 + \norm{\theta_0}_2^2/(\pi v_n) \big \}^{m_n}  } e^{-\psi_1} d \psi_1 \\
& \geq \bigg(\frac{v_n}{v_n + C}\bigg)^{m_n} e^{- c_1 \norm{\theta_0}_2^2 },
\end{align*}
for some constant $C > 0$. Finally, using $(1 - x)^{1/x} \geq e^{-2}$ for all $x \in (0, 1/2)$ and $e^{-1/x}$ is an increasing function in $x > 0$, we have, 
\begin{align*}
\bigg(\frac{v_n }{v_n + C}\bigg)^{m_n} \geq e^{- \sqrt{m_n}/2}. 
\end{align*}
Hence, the integral is bounded below by $e^{-(\sqrt{m_n} + c_1 \norm{\theta_0}_2^2)/2}$, resulting in
\begin{align}\label{eq:lg_simp_lb}
\bbP(\norm{\theta - \theta_0}_2 \leq r_n) \geq \frac{C_2 \xi_n v_n }{(n/2-1)}  \, \sqrt{\frac{v_n}{v_n +  \norm{\theta_0}_2^2}} \,  \, e^{-(\sqrt{m_n} + c_1 \norm{\theta_0}_2^2)/2}.
\end{align}

Thus, finally, noting that $u_n \to \infty$, 
\begin{align*}
\frac{\bbP(|| \theta - \theta_0 ||_2 < s_n )}{\bbP(|| \theta - \theta_0 ||_2 < r_n )}  \times e^{r_n^2}  \leq D \frac{w_n^{3/2}}{v_n} e^{C (\sqrt{m_n} + \sqrt{n} + \norm{\theta_0}_2^2)}  \, e^{- u_n \sqrt{m_n}  / 2} \to 0,
\end{align*} where $C, D > 0$ are constants. \qed

\subsection{\bf Proof of Theorem \ref{thm:lb_plugin} }
As before, we assume $\lambda = 1$ w.l.g., since it can be absorbed in the constant appearing the sequence $\tau_n$ otherwise. 
As in the proof of Theorem \ref{thm:lb_lg}, combine \eqref{eq:main_bd1} \& \eqref{eq:main_bd2} to obtain
\begin{eqnarray*}
\bbP(\norm{\theta - \theta_0} < t_n) &\leq& D^n \tau_n^{-n/2}  \int_{\psi} g(\psi)  \bigg\{ \int_{\sum x_j \leq w_n} \prod_{j=1}^n \frac{1}{\sqrt{x_j \psi_j}} \exp \bigg(-\frac{x_j + \theta_{0j}^2}{2\psi_j}\bigg) dx\bigg\} d\psi \\
&=& D^n \tau_n^{-n/2} w_n^{n/2} \int_{\sum x_j \leq 1}  \prod_{j=1}^n \frac{1}{\sqrt{x_j }} \bigg \{\prod_{j=1}^n \int_{\psi_j = 0}^{\infty} \frac{e^{-\psi_j}}{\sqrt{\psi_j}} \exp \bigg(-\frac{w_n x_j + \theta_{0j}^2}{2\tau_n \psi_j}  \bigg) d\psi_j\bigg\} dx,  \end{eqnarray*}
where $w_n = t_n^2$. Using the fact $\int_{0}^{\infty} \frac{1}{\sqrt{x}} \exp \big\{-\big(\frac{a}{x} + x \big)\big\} dx= \sqrt{\pi} e^{-2 \sqrt{a}}$, we obtain
\begin{align}
& \bbP(\norm{\theta - \theta_0} < t_n) \notag \\
& \leq D^n \pi^{n/2} \tau_n^{-n/2} w_n^{n/2}   \int_{\sum x_j \leq 1} \prod_{j=1}^n \frac{1}{\sqrt{x_j} } \exp \bigg\{-2\sqrt\frac{w_n x_j + \theta_{0j}^2}{2 \tau_n} \bigg\} dx \notag \\
& \leq  \bigg(\frac{D^2 \pi w_n}{\tau_n}\bigg)^{n/2} e^{-\frac{\sqrt{2} \norm{\theta_0}_1}{\sqrt{\tau_n}}} \int_{\sum x_j \leq 1} \prod_{j=1}^n \frac{1}{\sqrt{x_j} } dx = \bigg(\frac{D^2 \pi w_n}{\tau_n}\bigg)^{n/2} e^{-\frac{\sqrt{2} \norm{\theta_0}_1}{\sqrt{\tau_n}}} \frac{\Gamma(1/2)^n}{\Gamma(n/2+1)}, \label{eq:plug_in_ub}
\end{align}
where the second to third inequality uses $x_j \geq 0$ and the last integral follows from Lemma \ref{lem:dir_formula}. 
Along the same lines, 
\begin{align}
\bbP(\norm{\theta - \theta_0} < r_n) &\geq \bigg(\frac{D^2 \pi v_n}{\tau_n}\bigg)^{n/2} \int_{\sum x_j \leq 1} \prod_{j=1}^n \frac{1}{\sqrt{x_j} } \exp \bigg\{-2\sqrt\frac{v_n x_j + \theta_{0j}^2}{2 \tau_n} \bigg\} dx \notag \\
& \geq \bigg(\frac{D^2 \pi v_n}{\tau_n}\bigg)^{n/2} e^{-\frac{\sqrt{2} \norm{\theta_0}_1 + \sqrt{n v_n}}{\sqrt{\tau_n}}} \frac{\Gamma(1/2)^n}{\Gamma(n/2+1)},  \label{eq:plug_in_lb}
\end{align}
where $v_n = r_n^2/4$. From the second to third equation in the above display, we used $\sqrt{a+b} \leq \sqrt{a} + \sqrt{b}$ and $\sum_{j=1}^n  \sqrt{x_j} \leq \sqrt{n}$ by Cauchy-Schwartz inequality if $x \in \Delta^{n-1}$. Thus, from \eqref{eq:plug_in_ub} \& \eqref{eq:plug_in_lb}, the ratio in \eqref{eq:lb_ratio} can be bounded above as:
\begin{align*}
\frac{\bbP(\norm{\theta - \theta_0} < t_n)}{ \bbP(\norm{\theta - \theta_0} < r_n)}  \leq \bigg(\frac{w_n}{v_n}\bigg)^{n/2} e^{\sqrt{2 v_n n/\tau_n}}.
\end{align*}
Choose $t_n = s_n$, $r_n = 2 \sqrt{2} s_n$ so that $v_n = 2 w_n = 2 q_n \log(n/q_n)$ and $(w_n/v_n)^{n/2} = e^{- C n}$. Clearly $v_n n /\tau_n \leq C n q_n (\log n)^2$ and hence, $e^{\sqrt{2 v_n n/\tau_n}} = o(e^{C n})$ by assumption. Thus, the right hand side of the above display $\to 0$, proving the assertion of the Theorem. 
\subsection*{{\bf Proof of Theorem \ref{thm:kD_conc}}}

First, we will state a more general result on the concentration of $\mbox{DL}_{1/n}(\tau)$ when $\tau 
\sim \mbox{Exp}(\lambda)$. The result follows from a straightforward modification of  Lemma 4.1 in  \cite{pati2012posterior} and the detailed proof is omitted here.  Assume $\delta = t_n/(2n)$. 
For fixed numbers $0 < a <b < 1$, let  
\begin{eqnarray*}
\eta_n =  1-  \frac{(n- q_n) \delta}{2q_n \log (n/q_n)} - \frac{(q_n -1)b}{2q_n} , \quad \xi_n=  1- \frac{(q_n -1)a}{4q_n}.
\end{eqnarray*}
Also, without loss of generality assume that $\{1\} \subset S_0 = \mbox{supp}(\theta_0)$, i.e., $\theta_{01} \neq 0$. Let $S_1 = S_0 \backslash \{1\}$. 
If $\theta_0 \in l_0 (q_n; n)$, it follows that 
\begin{eqnarray*}
P( \norm{\theta - \theta_0} < t_n)  \geq  C \, \bbP( \tau \in [2q_n, 4q_n]) \, A_n B_n, 
\end{eqnarray*}
where $C$ is an absolute constant, and
\begin{eqnarray*} 
A_n &=&  \exp\bigg\{ -q_n \log 2 - \sum_{ j \in S_1} \frac{\abs{\theta_{0j}}}{a} - \theta_{01}/ (1-b)\eta_n\bigg\}  \\
B_n &=& \bigg[ 1 - \exp\bigg\{- \frac{t_n}{\sqrt{2q_n} b} \bigg\}  \bigg]^{q_n -1} \bigg[ 1- \exp\bigg\{- \frac{t_n}{\sqrt{2q_n} (1-a/8)\xi_n}\bigg\} \bigg]. \\
\end{eqnarray*}
In our case, $\abs{S_0} = 1$, $\theta_{01} = \sqrt{\log n}$, and
$t_n = n^{\delta/2}$.  Hence $A_n$ is a constant, $B_n = \exp \{- K_1 \sqrt{\log n} \}$ 
for some constant $K_1$ and 
$P( \tau \in [2q_n, 4q_n])$ is also a constant.  Hence,  
under the assumptions of  Theorem \ref{thm:kD_conc}, $
P( \norm{\theta - \theta_0} < t_n)  \geq  \exp \{ -C \sqrt{\log n}   \}$.


%

\appendix 
\section*{Appendix}
%
\subsection*{\bf Proof of Proposition \ref{propWG}}

When $a = 1/n$, $\phi_j \sim \mbox{Beta}(1/n,1-1/n)$ marginally. Hence, the marginal distribution of $\theta_j$ given $\tau$ is proportional to
\begin{align*}
\int_{\phi_j=0}^1 e^{- \abs{\theta_j}/(\phi_j \tau)} \bigg(\frac{\phi_j}{1 - \phi_j} \bigg)^{1/n} \phi_j^{-2} d \phi_j. 
\end{align*}
Substituting $z = \phi_j/(1 - \phi_j)$ so that $\phi_j = z/(1+z)$, the above integral reduces to 
\begin{align*}
e^{-\abs{\theta_j}/\tau} \int_{z=0}^{\infty} e^{- \abs{\theta_j}/(\tau z)} z^{-(2 - 1/n)} dz \propto e^{- \abs{\theta_j}/\tau} \abs{\theta_j}^{1/n-1}.
\end{align*}
In the general case, $\phi_j \sim \mbox{Beta}(a, (n-1)a)$ marginally. Substituting $z = \phi_j/(1 - \phi_j)$ as before, the marginal density of $\theta_j$ is proportional to
\begin{align*}
e^{-\abs{\theta_j}/\tau} \int_{z=0}^{\infty} e^{- \abs{\theta_j}/(\tau z)} z^{-(2 - a)}  \bigg(\frac{1}{1+z}\bigg)^{na-1} dz. 
\end{align*}
The above integral can clearly be bounded below by a constant multiple of 
\begin{align*}
e^{-\abs{\theta_j}/\tau} \int_{z=0}^{1} e^{- \abs{\theta_j}/(\tau z)} z^{-(2 - a)}  dz. 
\end{align*}
Resort to Lemma \ref{lem:incomplete_gamma} to finish the proof.
\subsection*{{\bf Proof of Lemma \ref{lem:incomplete_gamma}}}

Using a simple change of variable, 
\begin{align*}
&\int_{\tau=0}^{1} \tau^{-n/2} e^{-a_n/(2 \tau)} d\tau = \int_{z=1}^{\infty} z^{n/2-2} e^{-a_nz/2} dz  \\
& \bigg( \frac{2}{a_n} \bigg)^{n/2-1}\int_{t=a_n/2}^{\infty} t^{a_n/2-2} e^{-t} dt = \bigg( \frac{2}{a_n} \bigg)^{n/2-1} \bigg[ \Gamma(n/2-1) - \int_{t=0}^{a_n/2} t^{n/2-2} e^{-t} dt \bigg]
\end{align*}
Noting that $\int_{t=0}^{a_n/2} t^{n/2-2} e^{-t} dt \leq a_n^{n/2-1}/(n/2-1)$ and $a_n \leq n/(2e)$ by assumption, the last entry in the above display can be bounded below by
\begin{align*}
\bigg( \frac{2}{a_n} \bigg)^{n/2-1} \Gamma(n/2-1) \bigg[1 - \frac{a_n^{n/2-1}}{\Gamma(n/2)} \bigg] \geq \bigg( \frac{2}{a_n} \bigg)^{n/2-1} \bigg[1 - \frac{\{n/(2e)\}^{n/2-1}}{\Gamma(n/2)} \bigg] .
\end{align*}
Let $\xi_n = 1 - \{n/(2e)\}^{n/2-1}/\Gamma(n/2)$. Using the fact that $\Gamma(m) \geq \sqrt{2 \pi} \, m^{m-1/2} e^{-m}$ for any $m > 0$, one has 
$\Gamma(n/2) \geq C \{n/(2e)\}^{n/2-1} \sqrt{n}$ with $C = e \sqrt{\pi}$. Hence, $(1 - \xi_n) \leq C/\sqrt{n}$ for some absolute constant $C > 0$. 

\subsection*{{\bf Proof of Lemma \ref{lem:dickey_integral} } }


Let $s = n/2$, $T = \sum_{j=1}^n q_j x_j + q_0$ and $q_j' = (q_j + q_0)$.  Then, the multiple integral in Lemma \ref{lem:dickey_integral} equals
\begin{align}
& \int_{\sum_{j=1}^n x_j \leq 1} \frac{1}{T^{s-1}} \prod_{j=1}^n x_j^{\alpha_j -1} dx  = \int_{\sum_{j=1}^n x_j \leq 1} \frac{1}{T^{s-1}}  \prod_{j=1}^n \bigg( \frac{q_j'x_j}{T}\bigg)^{\alpha_j -1} \prod_{j=1}^n 
\bigg( \frac{T}{q_j'}\bigg)^{\alpha_j -1} dx \nonumber \\
& =  \prod_{j=1}^n \bigg( \frac{1}{q_j'}\bigg)^{\alpha_j -1} 
\int_{\sum_{j=1}^n x_j \leq 1} T^{1-n} 
\prod_{j=1}^n \bigg( \frac{q_j'x_j}{T}\bigg)^{\alpha_j -1} dx \label{eq:I}
\end{align}
Now, we make a change of variable from $x$ to $z$, with $z_j = q_j' x_j/T$ for $j=1 \ldots, n$. Clearly, $z$ also belongs to the simplex $\Delta^{(n-1)}$. Moreover, letting $z_{n+1} = 1- \sum_{j=1}^n z_j$, one has $z_{n+1} = q_0 x_{n+1}/T$, where $x_{n+1} = 1- \sum_{j=1}^nx_j$. Thus, by composition rule, 
\begin{align}\label{eq:T}
T = \frac{x_1}{\frac{z_1}{q_1'}} = \cdots =  \frac{x_n}{z_n/ q_n} = 
\frac{x_{n+1}}{ z_n/q_0} = \frac{1}{z_1/q_1' + \cdots + z_n/ q_n' + z_{n+1}/q_0}
\end{align}  
Let $J =  \big( \frac{\partial x_j}{\partial z_l}\big)_{jl}$ be the Jacobian of the transformation and $H =  \big( \frac{\partial x_j}{\partial z_l}\big)_{jl} = J^{-1}$. 
Then, 
\begin{eqnarray*}
H_{jl} = 
\begin{cases}
\frac{q_j'}{T^2} (T - q_j X_j) \,\mbox{if} \, l=j \\
\frac{-q_j'X_j}{T^2} q_l  \,\mbox{if} \, l\neq j
\end{cases}
\end{eqnarray*}
Clearly, $\abs{H} = \abs{H_1}\prod_{j=1}^n \frac{q_j'}{T^2} $ with $H_1 = T \mathrm{I}_n - x q^{\T}$, where $q = (q_1, \ldots, q_n)^{\T}$ and $\abs{A}$ denotes the determinant of a square matrix $A$. Using a standard result for determinants of rank one perturbations, one has 
$\abs{H_1} =  T^n \abs{\mathrm{I}_n - \frac{1}{T} x q^{\T}} = T^n \big( 1 - \frac{q^{\T} x}{T} \big) = q_0 T^{n-1} $, implying $\abs{H} =  (q_0 T^{n-1}) \prod_{j=1}^n \frac{q_j'}{T^2} = \frac{q_0}{T^{n+1}}  \prod_{j=1}^n q_j'$. 
Hence the Jacobian of the transformation is 
\begin{eqnarray*}
\abs{J} = \frac{T^{n+1}} {r \prod_{j=1}^n q_j'}, 
\end{eqnarray*}
so that the change of variable in (\ref{eq:I}) results in
\begin{align}
& \prod_{j=1}^n \bigg(\frac{1}{q_j'}\bigg)^{\alpha_j -1} \frac{1}{q_0 \prod_{j=1}^n q_j'} \int_{\sum_{j=1}^n z_j \leq 1} \bigg\{ \prod_{j=1}^n z_j ^{\alpha_j -1} \bigg\} T^2 dz  \nonumber \\
&= q_0 \prod_{j=1}^n \bigg(\frac{1}{q_j'}\bigg)^{\alpha_j} \int_{\sum_{j=1}^n z_j \leq 1} \bigg\{ \prod_{j=1}^n z_j ^{\alpha_j -1} \bigg\} \frac{T^2}{q_0^2} dz \nonumber  \\
&=q_0 \prod_{j=1}^n \bigg(\frac{1}{q_j'}\bigg)^{\alpha_j} \int_{\sum_{j=1}^n z_j \leq 1} \frac{1}{(\nu_1z_1 + \cdots + \nu_nz_n + z_{n+1})^2} \bigg\{ \prod_{j=1}^n z_j ^{\alpha_j -1} \bigg\} dz \label{eq:I2} 
\end{align}
where $v_j = \frac{q_0}{q_j + q_0} = \frac{q_0}{q_j'}$.  
Now, the expression in (\ref{eq:I2}) clearly equals 
\begin{eqnarray} \label{eq:I3}
q_0 \prod_{j=1}^n \bigg(\frac{1}{q_j'}\bigg)^{\alpha_j} \frac{\prod_{j=1}^n\Gamma(\alpha_j)}{\Gamma(s+1)}  \bbE \bigg\{ \nu_1 Z_1 + \cdots +\nu_n Z_n + Z_{n+1} \bigg\}^{-2}, 
\end{eqnarray}
where $(Z_1, \ldots, Z_n) \sim \mbox{Dir}(\alpha_1, \ldots, \alpha_n, 1)$. A profound result in \cite{dickey1968three} shows that expectations of functions of Dirichlet random vectors as above can be reduced to the expectation of a functional of  univariate Beta random variable: \\
{\bf Dickey's formula \cite{dickey1968three}}: Let $(Z_1, \cdots, Z_n) \sim \mbox{Dir}(\beta_1, \cdots, \beta_n , \beta_{n+1})$ and $Z_{n+1} = 1- \sum_{i=1}^nZ_i$.  Suppose 
$a < \sum_{j=1}^{n+1} \beta_j$. Then, for $\nu_j > 0$, 
\begin{eqnarray*}
\bbE \bigg[ \sum_{j=1}^{n+1} \nu_j Z_j  \bigg]^{-a} = \bbE 
\prod_{j=1}^{n+1} \frac{1}{ \{ \nu_j + X ( 1- \nu_j) \}^{\alpha_j }}, 
\end{eqnarray*}
where $X \sim \mbox{Beta}(b, a)$ with $b = \sum_{j=1}^{n+1} \beta_j -a$.

Applying Dickey's formula with $\beta_j = \alpha_j = 1/2$ for $j = 1, \ldots, n$, $\beta_{n+1} = 1$ and $a=2$ (so that $b= \frac{n}{2} +1 -2 = \frac{n}{2} -1$), \eqref{eq:I3} reduces to\begin{eqnarray}\label{eq:I4}
q_0 \prod_{j=1}^n \bigg(\frac{1}{q_j'}\bigg)^{\alpha_j} \frac{\prod_{j=1}^n\Gamma(\alpha_j)}{\Gamma(s+1)} \bbE  \prod_{j=1}^n \frac{1}{\big\{ \nu_j + X(1- \nu_j)\big\}^{\alpha_j}}
\end{eqnarray}
where $X \sim \mbox{Beta}(b, a)$ with density $f(x) = (n/2)(n/2-1) x^{n/2-2} (1 - x)$ for $x \in (0, 1)$.  
Hence, (\ref{eq:I4}) finally reduces to 
\begin{eqnarray*}
q_0 \bigg( \frac{n}{2} -1\bigg)
 \frac{\Gamma(1/2)^n}{\Gamma(n/2)} \int_{0}^1 \frac{ x^{n/2 -2} (1-x)}{\prod_{j=1}^n (q_j x + q_0)^{\alpha_j}} dx
\end{eqnarray*} \qed

\subsection*{\bf Error function bounds}
Let $\mbox{erfc}(x) = \frac{2}{\sqrt{\pi}}\int_{x}^{\infty} e^{-t^2} dt$ denote the complementary error function; clearly $\mbox{erfc}(x) = 2[ 1 - \Phi(\sqrt{2} x) ]$, where $\Phi$ denotes the standard normal c.d.f. A standard inequality (see, for example, Formula 7.1.13 in \cite{abramowitz1965handbook}) states
\begin{align}\label{eq:erfc_bd}
\frac{2}{x + \sqrt{x ^2+ 2}} \leq \sqrt{\pi} e^x \mbox{erfc}(\sqrt{x}) \leq \frac{2}{x + \sqrt{x + 4/\pi}}
\end{align}
Based on \eqref{eq:erfc_bd}, we show that
\begin{align}
& \sqrt{\pi} e^x \mathrm{erfc}(\sqrt{x}) \leq \frac{1}{\sqrt{x + 1/\pi}} \label{eq:erfc_ub} \\
& \sqrt{\pi} e^x \mathrm{erfc}(\sqrt{x}) \geq \bigg\{ \frac{1}{\sqrt{x}}  \bigg\}^{1 + \delta} \label{eq:erfc_lb}
\end{align}
where \eqref{eq:erfc_lb} holds for any $\delta > 0$ provided $x \geq 2$. 

In view of \eqref{eq:erfc_bd}, to prove \eqref{eq:erfc_ub} it is enough to show that $2 \sqrt{x + 1/\pi} \leq \sqrt{x} + \sqrt{x + 4/\pi}$, which follows since:  
\begin{eqnarray*}
(\sqrt{x} + \sqrt{x + 4/\pi})^2 - 4(x + 1/\pi) = 2x + 2\sqrt{x} \sqrt{x + 4/\pi} - 4x \geq 0 . 
\end{eqnarray*}
To show \eqref{eq:erfc_lb}, we use the lower bound for the complementary error function in \eqref{eq:erfc_bd}. 
First, we will show that for any $\delta > 0$, $x+ \sqrt{x^2 + 2} \leq 2x^{1+ \delta}$ if $x \geq 2$. Noting that if $x \geq 2$
\begin{eqnarray*}
x^{2+ 2\delta} - x^2 = x^2(x^{2\delta} -1) = x^2(x-1) (1+ x+ \cdots + x^{2\delta}) \geq 2.
\end{eqnarray*}  
Hence $\sqrt{x^2 + 2} \leq x^{1+\delta}$ if $x \geq 2$, showing that $x+ \sqrt{x^2 + 2} \leq 2x^{1+ \delta}$ if $x \geq 2$.
Thus, we have, for $x \geq 2$ and any $\delta > 0$, 
\begin{eqnarray*}
\sqrt{\pi} e^x \mbox{erfc}(\sqrt{x}) \geq \bigg(\frac{1}{\sqrt{x}}\bigg)^{1+\delta} .  
\end{eqnarray*} \qed

\bibliographystyle{plain}
\bibliography{cov_refs}
\end{document}